\documentclass[fleqn]{article}

\usepackage{amsmath,amssymb,amsthm,tikz,enumerate,stmaryrd}
\usepackage[f]{esvect}

\title{Certified $\Sigma_1$-sentences}
\author{Taishi Kurahashi\footnote{Email: kurahashi@people.kobe-u.ac.jp}
\footnote{Graduate School of System Informatics, Kobe University, Japan.} \ 
and Albert Visser\footnote{Email: a.visser@uu.nl}\ \footnote{Philosophy, Faculty of Humanities, Utrecht University, The Netherlands} 
}
\date{}

\theoremstyle{plain}
\newtheorem{thm}{Theorem}[section]
\newtheorem{lem}[thm]{Lemma}
\newtheorem{prop}[thm]{Proposition}
\newtheorem{cor}[thm]{Corollary}

\theoremstyle{definition}
\newtheorem{defn}[thm]{Definition}
\newtheorem{ex}[thm]{Example}
\newtheorem{rem}[thm]{Remark}
\newtheorem{prob}[thm]{Problem}

\let\conTheta\Theta

\renewcommand{\Theta}{\varTheta}
\renewcommand{\Phi}{\varPhi}
\renewcommand{\Psi}{\varPsi}
\renewcommand{\Xi}{\varXi}
\renewcommand{\Omega}{\varOmega}
\renewcommand{\Gamma}{\varGamma}

\newcommand{\PA}{\mathsf{PA}}
\newcommand{\PR}{\mathsf{Pr}}

\newcommand{\gn}[1]{\ulcorner{#1}\urcorner}
\newcommand{\sgn}[1]{\lceil{#1}\rceil}

\newcommand{\num}[1]{\underline{#1}}
\newcommand{\TR}{\vartriangleright}
\newcommand{\NTR}{\ntriangleright}
\newcommand{\mc}[1]{\mathcal {#1}}
\newcommand{\W}[1]{\mathsf{W}_{#1}}

\newcommand{\N}{\mathbb{N}}

\newcommand{\sq}{\mathsf{cert}}
\newcommand{\RR}{\mathsf{R}}
\newcommand{\LA}{\mathbb{L}_{\sf a}}
\newcommand{\LAP}{\mathbb{L}_{\sf ap}}
\newcommand{\RP}{{\sf R}_{0{\sf p}}}

\newcommand{\s}[1]{\mathsf{s}\hspace*{0.02cm} #1}
\newcommand{\sbr}[1]{\mathsf{s}\hspace*{0.01cm}(#1)}
\newcommand{\Ax}[1]{{\sf A{#1}}}
\newcommand{\verz}[1]{\{ #1 \}}

\renewcommand{\iff}{\leftrightarrow}
\newcommand{\grullet}{{\textcolor{gray}{\ensuremath{\bullet}}}}

\newcommand{\certh}{Certified Extension Theorem}
\newcommand{\cert}{Certified Extension}
\newcommand{\virtu}{certified}
\newcommand{\virt}{{\sf cert}}
\newcommand{\Vness}{Certification}
\newcommand{\vness}{certification}
\newcommand{\idax}[1]{\mathfrak{id}_{#1}}
\newcommand{\idap}{\idax{\sf ap}}
\newcommand{\gener}{$\RP$-sourced}
\newcommand{\certo}[1]{{[ #1]}}
\newcommand{\certt}[1]{{\llbracket #1 \rrbracket}}
\newcommand{\virtob}{{\mathfrak v}}
\newcommand{\mf}[1]{{\mathfrak {#1}}}
 \newcommand{\nrhd}{\mathrel{\not\! \rhd}}
 \renewcommand{\leq}{\leqslant}
 \newcommand{\bles}{\mathbin{<}}
\newcommand{\bleq}{\mathbin{\leq}}

\newcommand{\sline}{\raise-0.3ex\hbox{$\hbox{--}\kern-0.84ex\raise0.45ex\hbox{$\hbox{\scalebox{0.3}{\bf /}}\kern-0.37ex\hbox{\scalebox{0.3}{\bf /}}$}$}}
\newcommand{\slinei}{\raise-0.3ex\hbox{$\hbox{--}\kern-0.84ex\raise0.45ex\hbox{$\hbox{\scalebox{0.3}
{\bf \textbackslash}}\kern-0.37ex\hbox{\scalebox{0.3}{\bf \textbackslash}}$}$}}
\newcommand{\jumpneq}{\mathrel{\jump_{\hspace*{-0.235cm}{}_{\kern0.08ex \sline}}\hspace*{0.09cm}}}
\newcommand{\jumpbneq}{\mathrel{\jumpb_{\hspace*{-0.235cm}{}_{\kern0.08ex \slinei}}\hspace*{0.09cm}}}
\newcommand{\dashvneq}{\mathrel{\dashv_{\hspace*{-0.235cm}{}_{\kern0.2ex \slinei}}\hspace*{0.09cm}}}
\newcommand{\vdashneq}{\mathrel{\vdash_{\hspace*{-0.235cm}{}_{\kern0.3ex \sline}}\hspace*{0.09cm}}}
\newcommand{\lhdnneq}{\mathrel{\lhd_{\hspace*{-0.27cm}{}_{\kern0.2ex \slinei}}\hspace*{0.09cm}}}
\newcommand{\rhdnneq}{\mathrel{\rhd_{\hspace*{-0.27cm}{}_{\kern0.3ex \sline}}\hspace*{0.09cm}}}
\newcommand{\cothe}[1]{{\sf coTh}_{#1}}

\definecolor{uured}{cmyk}{0.2,1,0.9,0.1}
\definecolor{uublue}  {cmyk}{0.9,0.55,0,0}
\definecolor{uugreen}{cmyk}{1,0,0.75,0}
\definecolor{bazaar}{rgb}{0.6, 0.47, 0.48}

\begin{document}

\maketitle

\begin{abstract}
In this paper, we study the employment of $\Sigma_1$-sentences with certificates, i.e., $\Sigma_1$-sentences where
a number of principles is added to ensure that the witness is sufficiently number-like. We develop certificates
in some detail and illustrate their use by reproving some classical results and proving some new ones. 
An example of such a classical result is Vaught's theorem of the strong effective inseparability of $\RR_0$. 

We also develop the new idea of a theory being \emph{${\sf R}_{0{\sf p}}$-sourced}. Using this notion, we can  transfer a number of
salient results from $\RR_0$ to a variety of other theories.
\end{abstract}

\section{Introduction}
In this paper, we study \emph{certificates}. These are \emph{theories-of-a-number} with a free parameter 
for the number in question, or, more precisely, for the \emph{number-like object}.
In other words, certificates specify a property of a number. This property is roughly that the object specified is sufficiently like a number.
A salient property of theories-of-a-number is that they have finite models.

The main focus of this paper is on certificates as a tool to metamathematical results.
Thus, the paper can be viewed as a study of certificates as a method.
We develop one specific  certificate and provide the necessary lemmas for its employment. We apply the
certificate to, possibly non-standard, witnesses of $\Sigma^0_1$-sentences. This use of the
certificate is in constant interaction with the salient theories $\RR_0$ and $\RR$. Our presentation provides more detail than
previous presentations, so that many subtleties of what is going on become clearly visible here for the first time.

We  extend the classical results obtained by the use of certificates by defining a wider class of theories, the
\emph{$\RR_{0{\sf p}}$-sourced theories}. These theories behave in some important respects like the salient theory $\RR_0$.

The paper presents a number of applications of the use of certificates, which are important in themselves, but also serve to
illustrate the use of the method well. These are:
\begin{itemize}
    \item 
    Certain theorems by Cobham and Vaught, the contents of which are explained in Section~\ref{cova} below. We introduce these results in our
    preparatory Section~\ref{cova}. The detailed treatment then will be
     in Section~\ref{vaughtagain}. Our version of the second Vaught theorem is a generalization to the $\RR_{0{\sf p}}$-sourced
    case.
    \item 
    A variety of results concerning degree structures of interpretability. These results are in Section~\ref{degrees}.
\end{itemize}

Theories-of-a-number have counterparts for various other data-types, like sets, multi-sets, sequences, and strings.
The alternative that is developed and used is theories-of-a-finite-set. 
Here is a list of 
examples of uses of theories-of-a-number and their kin that we noted. We do not have  any pretense of completeness here.
\begin{itemize}
\item 
In \cite{frie:inte07}, Harvey Friedman uses theories-of-a-number to prove the density of the interpretability degrees of finitely 
axiomatized theories. This result was proved earlier, by another method, in \cite{myci:latt90}. We present a version of the result
in Theorem~\ref{App4} and Corollary~\ref{locdens}.
\item 
In \cite{viss:vaug12}, Albert Visser uses theories-of-a-number to reprove (and improve) Vaught's result \cite{vaug:axio67} that every c.e.~Vaught theory is 
axiomatizable by a scheme.
    \item 
    In \cite{pakh:weak19}, Fedor Pakhomov uses
    theories-of-a-finite-set to construct an $\RR$-like set theory that proves its own consistency.
    \item 
    In \cite{pakh:ques19}, Fedor Pakhomov and Albert Visser show the following. Consider a finitely axiomatised extension $A$ of c.e.~theory $U$ in a possibly extended signature.
Suppose $A$ is conservative over $U$. Then, there is a conservative extension $B$ of $U$ in the signature of $A$, such that $A\vdash B$
and $B \nvdash A$. They use theories-of-a-finite-set to prove this result.
\end{itemize}

\subsection{Plan of the Paper}
The plan of the paper is as follows. In Section~\ref{cova}, we give a first presentation of both certain results by
Cobham and Vaught and a preliminary explanation of the
use of certificates. Section~\ref{bade} fixes some basic definitions and provides pointers to relevant literature. 
In Section~\ref{ce}, we develop the basic facts about certificates of a $\Sigma^0_1$-witness and the theories 
$\RR_0$ and $\RR$. In Section~\ref{hoempapasmurf}, we generalize the results of Section~\ref{ce} by replacing $\RR_0$ by
theories from a class that has $\RR_0$ as its source. We provide examples to illustrate that many salient theories are
in that class. Section~\ref{wicomp} provides basic facts about witness comparison, which is an important tool that we use
in the subsequent sections. The section is needed since the interaction between certificates and witness comparisons is
somewhat delicate. In Section~\ref{vaughtagain}, we apply the methods developed in the previous sections to prove
two theorems due to Vaught. Finally, in Section~\ref{degrees}, we apply these methods to prove various results about degrees
of interpretability. Sections~\ref{vaughtagain} and \ref{degrees} can be read independently of each other.

\subsection{History of the Paper}
The present paper succeeds and replaces Taishi Kurahashi's earlier preprint \emph{Incompleteness and undecidability of theories consistent with $\RR$}.
The materials of the preprint are contained in the present paper.

\section{Theorems of Cobham and Vaught}\label{cova}

In this section, we give a presentation of current status of certain theorems of Cobham and Vaught.
After some preparatory work, we take up this thread again in Section~\ref{vaughtagain}.

Let $\LA$ be the signature $\{0, \s, +, \times, \leq\}$ of first-order arithmetic. 
Let $\omega$ denote the set of all natural numbers. 
For each $n \in \omega$, the numeral $\s{\cdots \s{0}}$ ($n$ applications of $\s{}$) for $n$ is denoted by $\num{n}$. 
A central object of study in this paper is the $\LA$-theory $\RR$ introduced by Tarski, Mostowski and Robinson in \cite{TMR53}. 

\begin{defn}[The theory $\RR$]
The theory $\RR$ is axiomatized by the following sentences: For $m, n \in \omega$, 
\begin{enumerate}[{\sf R}1.]
	\item  $\num{m} + \num{n} = \num{m + n}$
	\item  $\num{m} \times \num{n} = \num{m \times n}$
	\item  $\num{m} \neq \num{n}$ (if $m \neq n$)
	\item  $\forall x\, \bigl(x \bleq \num{n} \to \bigvee_{i \bleq n} x = \num{i}\bigr)$
	\item  $\forall x\, (x \leq \num{n} \lor \num{n} \leq x)$
\end{enumerate}
\end{defn}

It was proved in \cite{TMR53} that the original $\RR$ is essentially undecidable, that is, every consistent extension of $\RR$ is undecidable. 
Here, we comment on the difference between our formulation and the original one of the theory $\RR$. 
The signature adopted in \cite{TMR53} does not contain the symbol $\leq$ and the formula $x \leq y$ is introduced as the abbreviation for $\exists z \; z + x = y$. 
Our signature $\LA$ contains $\leq$ as a primitive symbol and our version of 
$\RR$ does not prove the equivalence between $x \leq y$ and $\exists z \; z + x = y$. 
So, our $\RR$ is strictly weaker than the original. 
Jones and Shepherdson \cite{JS83} pointed out that the essential undecidability of $\RR$
also holds without using the equivalence $x \leq y \leftrightarrow \exists z \; z + x = y$.

The $\LA$-theory $\RR_0$ is obtained from $\RR$ by replacing the axiom $\mathsf{R5}$ with the following $\mathsf{R5'}$: 
\begin{description}
	\item [$\mathsf{R5'}$] $\num{m} \leq \num{n}$ (if $m \leq n$)
\end{description}
Alternatively, we can axiomatize $\RR_0$ by dropping {\sf R5} alltogether and replacing {\sf R4} by:
\begin{description}
	\item [$\mathsf{R4'}$] $\forall x\, \bigl(x \bleq \num{n} \iff \bigvee_{i \bleq n} x = \num{i}\bigr)$
\end{description}
We make $\mathsf{R1}$-$\mathsf{4}$ plus $\mathsf{R5'}$ our official axiomatization.

\begin{rem}{\small
Our definitions of {\sf R} and ${\sf R}_0$ correspond to those in {\v{S}}vejdar's paper \cite{Sve08}. 

Vaught's definition of ${\sf R}_0$ in \cite{Vau62} is in the same spirit as ours, but still differs.
He presents the
theory in a relational format without identity. There is a mistake in Vaught's statement of the
axioms. The theory as given in the paper clearly has a decidable extension.
As possible repairs, one could add the axioms for the totality of the functions or replace, both in Axiom I and II, the second occurrence of $\to$ 
by $\iff$. 

Jones and Shepherdson in \cite{JS83} discuss both the version of {\sf R} with and without a defined relation $\leq$. 
They use ${\sf R}'_0$ for our ${\sf R}_0$.}
\end{rem}

It is easy to see that $\RR_0$ is a proper subtheory of $\RR$. We note that all the axioms of $\RR_0$ can be rewritten as $\Delta_0$-formulas.
This leads to a nice observation by V\'{\i}t\v{e}zslav {\v{S}}vejdar:

\begin{thm}[Cf.~{\v{S}}vejdar \cite{Sve08}]\label{SCompl}
An $\LA$-theory $T$ is $\Sigma_1$-complete if and only if $T \vdash \RR_0$. 
\end{thm}

Cobham observed that $\RR$ is interpretable in $\RR_0$, and that, hence,
$\RR_0$ is essentially undecidable (see Vaught \cite{Vau62} and Jones and Shepherdson \cite{JS83}).\footnote{It is also easy to prove
the essential undecidability of $\RR_0$ directly.} We give Cobham's interpretation in Section~\ref{vaughtagain}.

In the formulation of $\RR_0$, if $x \leq y$ is defined by $\exists z \; z + x = y$ as in \cite{TMR53}, 
rather than primitive as in the present paper,
then the axiom $\mathsf{R5'}$ is redundant since it can be derived from $\textsf{R1}$. 
On the other hand, note that, in our signature $\LA$, the theory obtained from $\RR_0$ by 
removing $\mathsf{R5'}$ has a complete consistent decidable extension (see Jones and Shepherdson \cite{JS83}).
 
We say that a theory $T$ is \textit{essentially hereditarily undecidable} if every $\LA$-theory consistent with $T$ is undecidable (cf.~\cite{Vis22}). 
It is shown in \cite{TMR53} that every finitely axiomatized essentially undecidable theory is also essentially hereditarily undecidable. 
Here, since the theory $\RR_0$ is not finitely axiomatizable, it is nontrivial whether $\RR_0$ is essentially hereditarily undecidable. 
In fact, there exists a computably axiomatized essentially undecidable theory having a decidable subtheory (cf.~Ehrenfeucht \cite{Ehr57} and Putnam \cite{Put57}). 
Then, Cobham proved the following interesting theorem. 

\begin{thm}[Cobham, see Vaught {\cite[1.5]{Vau62}}]\label{Cob}
The theory $\RR_0$ is essentially hereditarily undecidable.  
\end{thm}

A proof of Cobham's theorem was presented in Vaught \cite{Vau62}. 
Vaught also showed two strengthenings of Cobham's theorem. 

For each $i \in \omega$, let $\W{i}$ denote the c.e.~set with the index $i$. 
We say that a pair $(X, Y)$ of disjoint c.e.~sets is \textit{effectively inseparable} if for any $i, j \in \omega$, if 
$X \subseteq \W{i}$, $Y \subseteq \W{j}$, and $\W{i} \cap \W{j} = \emptyset$, then we can effectively 
find an element $x$ such that $x \notin \W{i} \cup \W{j}$.  
For each theory $T$, let $T_{\mathfrak{p}}$ and $T_{\mathfrak{r}}$ be the set of all theorems of $T$ and the set of all sentences refutable in $T$, respectively. 
We say that a consistent theory $T$ is \textit{strongly effectively inseparable} if the pair $(T_{\mathfrak{p}}, \emptyset_{\mathfrak{r}})$ is effectively inseparable (cf.~\cite{KV24}). 
The first strengthening is the following: 

\begin{thm}[Vaught {\cite[5.2]{Vau62}}]\label{Vau1}
The theory $\RR_0$ is strongly effectively inseparable. 
\end{thm}

In fact, Cobham's theorem follows easily from Theorem \ref{Vau1}. 
The second one is the following theorem that immediately implies Cobham's theorem, but no proof was presented in Vaught's paper.

\begin{thm}[Vaught {\cite[7.1]{Vau62}}]\label{Vau2}
For any c.e.~$\LA$-theory $U$, if $\RR_0 + U$ is consistent, then there exists a finitely axiomatized
$\LA$-theory $S$ extending $\RR_0$ such that $S + U$ is also consistent. 
\end{thm}

Recently, a more comprehensible proof of Cobham's theorem (Theorem \ref{Cob}) was also given in Visser \cite{Vis17}. 
\begin{defn}[Pure $\Delta_0$- and $\Sigma_1$-formulas]
Let $\varphi$ be an $\LA$-formula. 
\begin{itemize}
	\item We say that $\varphi$ is a \textit{pure $\Delta_0$-formula} if $\varphi$ is $\Delta_0$ and satisfies the following conditions: 
	\begin{enumerate}
		\item For any atomic formula of the form $t_1 \leq t_2$ contained in $\varphi$, terms $t_1$ and $t_2$ are both variables; 
		\item Every atomic formula of the form $t_1 = t_2$ contained in $\varphi$ is of one of the forms $x_0 = x_1$, $0 = x_0$, $\s{x_0} = x_1$, $x_0 + x_1 = x_2$, and $x_0 \times x_1 = x_2$, where $x_0$, $x_1$, and $x_2$ are variables. 		
	\end{enumerate}
	
	\item We say that $\varphi$ is a \textit{pure $\Sigma_1$-formula} if $\varphi$ is of the form $\exists \vv{x}\, \varphi_0(\vv{x})$, where $\varphi_0(\vv{x})$ is a pure $\Delta_0$-formula. 
	Here, the block $\vv{x}$ of quantifiers is allowed to be empty. 
\end{itemize}
\end{defn}

Our version of predicate logic does not contain the logical constants $\top$ and $\bot$ as primitive symbols. 
It is then shown that every pure $\Delta_0$-formula contains at least one free variable.
An effective procedure to obtain an equivalent pure $\Sigma_1$-formula for each $\Sigma_1$-formula is presented in \cite{Vis17}.

\begin{prop}[Visser \cite{Vis17}]\label{pure}
For any $\Sigma_1$-formula $\varphi(\vv{x})$, a pure $\Sigma_1$-formula $\varphi^\circ(\vv{x})$ satisfying the following conditions is effectively found: 
\begin{enumerate}
	\item $\N \models \forall \vv{x} \, (\varphi(\vv{x}) \leftrightarrow \varphi^\circ(\vv{x}))$, 
	\item $\forall \vv{x} \, (\varphi^\circ(\vv{x}) \to \varphi(\vv{x}))$ is logically valid. 
\end{enumerate}
\end{prop}

Here, we outline the proof of Cobham's theorem presented in \cite{Vis17}. 
At first, the finite $\LA$-theory $\mathsf{TN}$ (the theory of a number) is introduced. 
Then, for each pure $\Sigma_1$-sentence of the form $\exists \vv{x}\, \sigma_0(\vv{x})$, where $\sigma_0(\vv{x})$ is a pure $\Delta_0$-formula, let $\certo{\sigma}$ be the finite $\LA$-theory
\[
	\mathsf{TN} + \exists v\, \exists\vv{x} \bles v\, \sigma_0(\vv{x}). 
\]
Let $\sigma^\star$ be an $\LA$-sentence saying that there exists the least number $n$ such that the finite $\LA$-structure $\{0, 1, \ldots, n\}$ is a model of $\certo{\sigma}$. 
Then, the following three clauses hold for each pure $\Sigma_1$-sentence $\sigma$: \begin{enumerate}
	\item If $\N \models \sigma$, then $\RR_0 \vdash \sigma^\star$.  
	\item If $\N \models \neg\, \sigma$, then $\certo{\sigma} \vdash \RR_0$. 
	\item $(\RR_0 + \sigma^\star) \TR \certo{\sigma}$. 
\end{enumerate}
Here, $T \TR T'$ means that $T'$ is interpretable in $T$ (see Section \ref{bade} for the definition). 
Let $U$ be any $\LA$-theory such that $\RR_0 + U$ is consistent. 
We would like to show that $U$ is undecidable. 
We may assume that $U$ is a c.e.~theory. 
Then, the set $X : = \{\sigma \mid \sigma$ is a pure $\Sigma_1$  sentence and $\RR_0 + \sigma^\star + U$ is consistent$\}$ is $\Pi_1$-definable. 
Since the set $Y : = \{\sigma \mid \sigma$ is a true pure $\Sigma_1$-sentence$\}$ is not $\Pi_1$-definable, we have $X \neq Y$. 
By the first clause above, we have $Y \subsetneq X$, and hence $X \nsubseteq Y$. Then, we get a false pure $\Sigma_1$-sentence $\sigma$ such that $\RR_0 + \sigma^\star + U$ is consistent. 
By the second clause, we have $\certo{\sigma} \vdash \RR_0$, and thus the theory $\certo{\sigma}$ is essentially undecidable. 
Since $\certo{\sigma}$ is finite and $(\RR_0 + \sigma^\star) \TR \certo{\sigma}$ by the third clause, there exists a finite subtheory $S$ of $\RR_0 + \sigma^\star$ such that $S \TR \certo{\sigma}$. 
Then, $S$ is essentially undecidable and $S + U$ is consistent. 
Since $S$ is finite, we conclude that $U$ is undecidable.

However, it seems that the proof by Visser cannot be used to prove Vaught's theorems (Theorems \ref{Vau1} and \ref{Vau2}) as it is, 
because the notion of interpretability between theories is used in it. 


In this paper, we prove the following theorem using a modification of Visser's strategy.  

\begin{thm}[\certh]\label{CSS}
For each $\Sigma_1$-sentence $\sigma$, we can effectively find a sentence $\certo{\sigma}$ satisfying the following conditions: 
\begin{enumerate}[1.]
    \item $ \certo{\sigma} \vdash \sigma$. 
	\item If $\N \models \sigma$, then $\RR_0 \vdash \certo{\sigma}$. 
	\item If $\N \models \neg\, \sigma$, then $\certo{\sigma} \vdash \RR_0$. 
\end{enumerate}
We call such sentences $\certo{\sigma}$ \textit{certified $\Sigma_1$-sentences} for $\RR_0$. 
\end{thm}

We note that Visser's $[\sigma]$ does give us the analogues of (1) and (3) of Theorem~\ref{CSS}. However the analogue of (2) fails. We only get: if $\N \models \sigma$, then $\RR_0 \rhd [\sigma]$.

\section{Basic Definitions}\label{bade}
We only present a brief outline of the basic notions. For more detail,
we refer the reader
to, e.g., \cite[Appendix A]{Vis17}.

A \emph{theory} $U$ in this paper is a theory of predicate logic of finite signature $\Theta$.
A theory $U$ is given by a finite signature $\Theta$ and a set of axioms $X$ of the signature $\Theta$.

\begin{defn}\label{iddef}
The conjunction of the finitely many axioms for identity of $U$ is $\idax{\Theta}$ or $\idax{U}$.
The axioms of identity are officially part of the logic but at times we will treat them as if
they were part of the axioms of the theory.
\end{defn}

Let us fix an infinite sequence of variables $v_0,v_1,\dots$. Suppose $\Theta$ is a relational signature. 
\emph{A one-dimensional parameter-free translation $\tau:\Theta \to \Xi$} specifies a domain predicate
$\delta_\tau$, with at most $v_0$ free,  and, for an $n$-ary predicate symbol $R$ of $\Theta$ a $\Xi$-formula $R_\tau$,
 such that the free variables of $R_\tau$ are among $v_0,\dots,v_{n-1}$. 
 We treat identity as if it were a predicate from the signature rather than a logical predicate.
 We lift the translation
 to the full $\Theta$-signature as follows:
 \begin{itemize}
     \item $(R(\vv x))^\tau :=  R_\tau[\vv v:= \vv x]$. Here we assume an automatic mechanism of renaming variables in case of clashes.
     \item $(\cdot)^\tau$ commutes with the propositional connectives.
     \item $(\forall x \psi)^\tau := \forall x\, (\delta_\tau [v_0:= x] \to \psi^\tau)$.
     \item $(\exists x \psi)^\tau := \exists x\, (\delta_\tau [v_0:= x] \wedge \psi^\tau)$.
 \end{itemize}

 If $\Gamma$ is a set of $\Theta$-sentences, we write $\Gamma ^\tau$ for $\verz{\phi^\tau \mid \phi\in \Gamma}$.

We can extend translations to $m$-dimensional ones by translating a variable from the $\Theta$-signature to a sequence of
variables of length $m$ of the $\Xi$-signature. We can also allow parameters in our interpretations. 

We can define the identity translation and composition of translations in the obvious way.

An \emph{interpretation} ${\sf K}$ of $U$ in $V$ is a triple $(U,\tau, V)$, where $\tau$ is a translation from the
signature $\conTheta_U$ of $U$ to the signature $\conTheta_V$ of $V$. We demand that, for every $U$-sentence
$\phi$ such that $U \vdash \phi$, we have $V\vdash \phi^\tau$. We write ${\sf K}:U \lhd V$ or ${\sf K}: V \rhd U$ for
${\sf K}$ is an interpretation of $U$ in $V$. We write $U\lhd V$ or $V\rhd U$ for: there is a ${\sf K}$ such that ${\sf K}:U \lhd V$.

We have also the notion of \emph{local interpretability}. The theory $V$ \emph{locally interprets the theory $U$}, or $V \rhd_{\sf loc}U$, iff
for each finitely axiomatized subtheory $U_0$ of $U$, we have $V \rhd U_0$.

We will make use of the following operations on theories. Consider theories $U$ and $V$. 
\begin{itemize}
    \item 
    The theory $U \ovee V$ is defined as follows.
    The signature of $U \ovee V$ is the disjoint sum of the signatures of $U$ and $V$
    and, in addition, a fresh 0-ary predicate $P$.  The axioms of $U\ovee V$ are all $P \to \phi$, where
    $\phi$ is an axiom of $U$, plus all $\neg\, P \to \psi$, where $\psi$ is an axiom of $V$.
    \item 
    The theory $U \owedge V$ is defined as follows.
    The signature of $U \owedge V$ is the disjoint sum of the signatures of $U$ and $V$
    and, in addition, a fresh 1-ary predicate $\triangle$.  The axioms of $U\owedge V$ are the relativizations $\phi^\triangle$ of
    all axioms of $\phi$ of $U$ w.r.t. $\triangle$, plus  the relativizations $\psi^{\neg\,\triangle}$ of
    all axioms of $\psi$ of $V$ with respect to the complement of $\triangle$, plus axioms saying that neither $\triangle$ nor its complement are empty.
\end{itemize}

We have the following important properties.
\begin{thm}
    \begin{enumerate}[a.]
        \item 
        $(U\ovee V) \rhd W$ iff $U\rhd W$ and $V \rhd W$. 
        \item 
        $W \rhd (U\owedge V)$ iff $W\rhd U$ and $W \rhd V$. 
    \end{enumerate}
\end{thm}

Thus, $U \ovee V$ is (an implementation of) the \emph{infimum} of $U$ and $V$ in the degrees of interpretability and
$U \owedge V$ is (an implementation of) the \emph{supremum} of $U$ and $V$ in the degrees of interpretability.
There is something notationally awkward about representing an infimum by $\ovee$ and a supremum by $\owedge$.
This awkwardness is due to a legacy problem. In the boolean intuition the most informative element
is the bottom, where in the interpretation-ordering the most informative element is the top.
We follow the boolean intuition here, treating $\ovee$ as a kind of disjunction of theories and
$\owedge$ as a kind of conjunction.

\section{\cert}\label{ce}
We prove Theorem \ref{CSS}. 
Our proof is based on the ideas from \cite{Vis17}, but we exclude from the proof the use of $\sigma^\star$ and interpretability. 

Let $x < y$ be an abbreviation of $x \leq y \land x \neq y$. We define certain special elements as follows.

\begin{defn}[\Vness]\label{certification}
An element $v$ is \emph{\virtu}, or $\virt(v)$, if it satisfies the following formulas. These formulas together form the certificate.
\begin{enumerate}[{\sf A}1.]
	\item  $0 \leq v$
	\item  $\forall x < v\;\; \s x \leq v$ 
	\item  
 $\forall x \, (x\leq 0 \iff  x = 0)$
	\item  
 $\forall x \bles v\, \forall y\, (y\leq \s x \iff (y \leq x \vee y=\s x))$
	\item  $\forall x, y, z \bleq v\,\; \sbr{(x \times y) + z} \neq 0$
	\item  $\forall x, y, z, w \bleq v\,\; \sbr{(x \times y) + z} = \s w \to (x \times y) + z = w$
	\item  $\forall x, y \bleq v\; (x \times y) + 0 = x \times y$
	\item  $\forall x, y, z \bleq v\; (x \times y) + \s z = \sbr{(x \times y) + z} $
	\item  $\forall x \bleq v\; x \times 0 = 0$
	\item  $\forall x, y \bleq v\; x \times \s y = (x \times y) + x$
\end{enumerate}
\end{defn}

\noindent
We note that of the properties defining \vness, only $\Ax 3$ and $\Ax 4$ are not \emph{prima facie} $\Delta_0$. However,
we can rewrite $\Ax 3$ as $ \forall x\leq 0  \; x = 0 \wedge 0 \leq 0$ and we can rewrite $\Ax 4$ as 
\[\forall x \bles v\, (\forall y\bleq \s x \, (y \leq x \vee y=\s x) \wedge \forall y \bleq x\; y \leq \s x \wedge \s x \leq \s x).\]
So, modulo equivalence in predicate logic, \vness\ is $\Delta_0$.

The properties \Ax{5}--\Ax{8} look a bit different from the usual axioms in certificates. Their specific form is needed to prove
Lemma~\ref{LemAM}, which in its turn is needed to prove the negative atomic cases of Lemma~\ref{vrolijkesmurf}.

\begin{rem}
    We aimed to keep our definition of \vness\ reasonably simple. This has the advantage that it made clear that
    we can use a fairly light property. As we will see in Example~\ref{leergierigesmurf},
    it is possible to add all kinds of convenient properties to \vness\
     that preserve our intended application. An example of such a property is
    linearity of $\leq$ below $v$. 
\end{rem}

We say that $\sigma$ is a \textit{pure 1-$\Sigma_1$-sentence} if it is of the form
$\exists x\, \sigma_0(x)$, where $\sigma_0(x)$ is a pure $\Delta_0$-formula.

We strengthen Proposition \ref{pure} as follows:  

\begin{prop}\label{pure2}
For any $\Sigma_1$-sentence $\lambda$, a pure 1-$\Sigma_1$-sentence $\lambda^{\grullet}$ satisfying the following conditions can be effectively found: 
\begin{enumerate}
     \item $\N \models \lambda \leftrightarrow \lambda^\grullet$, 
	\item $\lambda^{\grullet} \to \lambda$ is logically valid. 
\end{enumerate}
\end{prop}
\begin{proof}
Consider any $\Sigma_1$-formula $\lambda$. 
By Proposition \ref{pure}, we can effectively find a pure $\Sigma_1$-formula ${\lambda}^{\circ}$ such that 
$\N \models \lambda \leftrightarrow \lambda^{\circ}$ and $\vdash \lambda^{\circ} \to \lambda$. 
Suppose that $\lambda^{\circ}$ is of the form $\exists \vv v\, \lambda_0(\vv v)$ for some pure $\Delta_0$-formula $\lambda_0(\vv v)$. 
Define $\lambda^{\grullet}$ to be the pure 1-$\Sigma_1$-sentence $\exists x\, \exists \vv v \leq x\, \lambda_0(\vv{v})$. 
Then, $\lambda^{\grullet}$ satisfies the conditions (1) and (2). 
\end{proof}

We are now ready to define certified $\Sigma_1$-sentences. 

\begin{defn}
Let $\sigma$ be a pure 1-$\Sigma_1$-sentence of the form $\exists x\, \sigma_0(x)$.
We define:
\[
\sigma^{\sq} :=	\exists x\, \bigl(\virt(x) \land \sigma_0(x)\bigr). 
\]
\end{defn} 
The following theorem is the heart of the technical part of our results.

\begin{thm}\label{QS}
Let $\sigma$ be a pure 1-$\Sigma_1$-sentence. Then:
\begin{enumerate}
    \item $\sigma^{\sq} \vdash \sigma$. 
    \item If $\N \models \sigma$, then $\RR_0 \vdash \sigma^{\sq}$. 
	\item If $\N \models \neg\, \sigma$, then $\sigma^{\sq} \vdash \RR_0$. 
\end{enumerate}
\end{thm}

For each $\Sigma_1$-sentence $\lambda$, we may define $[\lambda]$ to be the sentence $(\lambda^\grullet)^{\sq}$.
Then, Theorem \ref{CSS} immediately follows from Theorem \ref{QS} and Proposition \ref{pure2}. 

Before proving Theorem \ref{QS}, we investigate an $\LA$-model $\mc M$.
We assume that
\begin{itemize}
    \item [(\dag)]: $\virtob$ is a
designated \virtu\ element and $k$ is a natural number such that, for all $m < k$, we have $\mc M \models \num{m} \neq \virtob$.
\end{itemize}
Our first lemma is concerned with successor.

\begin{lem}[\dag]\label{Lem1}
For each $m \leq k$, we have $\mc M \models \num{m} \leq \virtob$. 
\end{lem}
\begin{proof}
We prove the lemma by induction on $m \leq k$. 
For $m = 0$, we have $\mc M \models 0 \leq \virtob$ by $\Ax{1}$. 
Suppose that the lemma holds for $m$ with $m + 1 \leq k$. 
Then, $\mc M \models \num{m} \leq \virtob$ by the induction hypothesis. 
Since $m< k$ and, hence, by Assumption (\dag), $\mc M \models \num{m} \neq \virtob$, we find $\mc M \models \num{m} < \virtob$.
So, by $\Ax{2}$, we may conclude that $\mc M \models \num{m+1} \leq \virtob$.
\end{proof}

\noindent
In the proofs of the following lemmas, we will use Lemma \ref{Lem1} without any reference. 

\begin{lem}[\dag]\label{LemL}
For any $m \leq k$, we have $\displaystyle \mc M \models \forall y\; \Bigl(y \leq \num{m} \leftrightarrow \bigvee_{l \leq m} y = \num{l} \Bigr)$. 
\end{lem}
\begin{proof}
We prove our lemma by induction on $m\leq k$. The case $m=0$ is precisely $\Ax{3}$. Suppose we have our equivalence for
$m$ with $m+1 \leq k$. We note that $m<k$, and hence $\mc M \models \num{m} < \virtob$. Thus,
we have, by $\Ax{4}$: 
\begin{align*}
	\mc M \models y \leq \num{m+1} & \;\,\iff\;\, y \leq \s{\num{m}} \\
		& \;\,\iff\;\, y \leq \num{m} \lor y = \s{\num{m}} \\
		& \;\,\iff\;\, \bigvee_{l \leq m} y = \num{l}  \lor y = \num{m+1} \\
		& \;\,\iff\;\, \bigvee_{l \leq m+1} y = \num{l}. \qedhere
\end{align*}
\end{proof}

\begin{lem}[\dag]\label{LemAM}
For any $m, n, p \leq k$, we have $\mc M \models (\num{m} \times \num{n}) + \num{p} = \num{(m \times n) + p}$. 
\end{lem}
\begin{proof}
We prove the lemma by induction on $n$. 
For $n = 0$, we prove $\mc M \models (\num{m} \times 0) + \num{p} = \num{p}$ by induction on $p$. 
For $p = 0$, we see that $\mc M \models (\num{m} \times 0) + 0 = \num{m} \times 0$ holds by $\Ax{7}$. 
Since $\mc M \models \num{m} \times 0 = 0$ by $\Ax{9}$, we obtain $\mc M \models (\num{m} \times 0) + 0 = 0$. 

Suppose that the statement holds for $p$ with $p + 1 \leq k$. 
Then, by $\Ax{8}$ and the induction hypothesis for $p$, we get
\[
	\mc M \models (\num{m} \times 0) + \num{(p+1)} = (\num{m} \times 0) + \s \num{p} = \sbr{(\num{m} \times 0) + \num{p}} = \s\num{p} = \num{p+1}. 
\]
We have proved that the lemma holds for $n = 0$. 

Suppose that the lemma holds for $n$ with $n + 1 \leq k$. 
We prove $\mc M \models (\num{m} \times (\num{n+1})) + \num{p} = \num{(m \times (n+1)) + p}$ by induction on $p$. 
For $p = 0$, by $\Ax{7}$, $\Ax{10}$, and the induction hypothesis for $n$, using that $m\leq k$,
\begin{eqnarray*}
\mc	M \models (\num{m} \times (\num{n+1})) + 0 & = & \num{m} \times \s\num{n} \\
& = & (\num{m} \times \num{n}) + \num{m}\\
	&  = & \num{(m \times n) + m} \\
 & = & \num{(m \times (n+1)) + 0}. 
\end{eqnarray*}

Assume that the statement holds for $p$ with $p + 1 \leq k$. 
By $\Ax{8}$ and the induction hypothesis for $p$, 
\begin{eqnarray*}
\mc	M \models (\num{m} \times (\num{n+1})) + (\num{p+1}) & = &  (\num{m} \times (\num{n+1})) + \s\num{p}\\
 &=& \sbr{(\num{m} \times (\num{n+1})) + \num{p}}\\
	& = & \sbr{(\num{m \times (n+1)) + p}}\\ &=& \num{(m \times (n+1)) + (p+1)}. \hspace*{2cm}\qedhere
\end{eqnarray*}
\end{proof}

\begin{lem}[\dag]\label{LemAM2}
Let $m, n \leq k$. 
\begin{enumerate}
	\item $\mc M \models \num{m} \times \num{n} = \num{m \times n}$, 
	\item $\mc M \models \num{m} + \num{n} = \num{m + n}$. 
\end{enumerate}
\end{lem}
\begin{proof}
1. By Lemma \ref{LemAM}, we get $\mc M \models (\num{m} \times \num{n}) + 0 = \num{m \times n}$. 
By $\Ax{7}$, we obtain $\mc M \models \num{m} \times \num{n} = \num{m \times n}$. 

2. If $k = 0$, then $m = n = 0$. 
By $\Ax{7}$, $\mc M \models (0 \times 0) + 0 = 0 \times 0$. 
Since $\mc M \models 0 \times 0 = 0$ by $\Ax{9}$, we obtain $\mc M \models 0 + 0 = 0$. 

If $k \geq 1$, then by Lemma \ref{LemAM} and Clause 1, we obtain 
\[
	\mc M \models \num{m} + \num{n} = (\num{m} \times \num{1}) + \num{n} = \num{m + n}. \qedhere
\]
\end{proof}

\begin{lem}[\dag]\label{LemN}
Suppose $m \leq k^2+k$ and $l\leq k$ and $m\neq l$. Then,
 $\mc M \models \num{m} \neq \num{l}$. 
\end{lem}

\begin{proof}
    We prove our result by induction on $l\leq k$. We will use the fact that (\ddag) every $m\leq k^2+k$, can be written as $(k\times m_0)+m_1$ for some $m_0,m_1\leq k$.
    
    For the base case, suppose $l =0$, $m \leq k^2+k$, and $m\neq l$.
    We have $m = m'+1$, so $\num{m} = \s \num{m'}$. So we are done by \Ax{5} in combination with (\ddag) and Lemma~\ref{LemAM}.
    
    We treat the successor case. Suppose $l= l'+1$ and we have the desired result for $l'$. Suppose also $m \leq k^2+k$, and $m\neq l$. In case $m=0$, we are  done by 
    \Ax{5} in combination with (\ddag) and Lemma~\ref{LemAM}. Suppose, $m = m' + 1$. We write $m'= (k\times m_0)+m_1$, where $m_0,m_1\leq k$.
    Suppose $\mc M \models \num{m} = \num{l}$. Then, by Lemma~\ref{LemAM}, we find $\mc M \models \s ((\num{k}\times \num{m_0})+\num{m_1}) = \s \num{l'}$.
    By \Ax{6}, we may conclude $\mc M \models (\num{k}\times \num{m_0})+\num{m_1} =  \num{l'}$. But this contradicts the induction hypothesis. 
\end{proof}

\begin{lem}[\dag]\label{vrolijkesmurf}
For any pure $\Delta_0$-formula $\varphi(x_0, \ldots, x_{i})$, and $n_0, \ldots, n_{i} \leq k$, 
 if $\N \models \varphi(\num{n_0}, \ldots, \num{n_i})$, then $\mc M \models \varphi(\num{n_0}, \ldots, \num{n_i})$. 
\end{lem}

\begin{proof}  
For any pure $\Delta_0$-formula $\varphi(x_0, \ldots, x_{i})$, and $n_0, \ldots, n_{i} \leq k$, we simultaneously 
prove the following two clauses by induction on the construction of $\varphi$:
\begin{enumerate}
    \item if $\N \models \varphi(\num{n_0}, \ldots, \num{n_i})$, then $\mc M \models \varphi(\num{n_0}, \ldots, \num{n_i})$, 
    \item if $\N \models \neg \, \varphi(\num{n_0}, \ldots, \num{n_i})$, then $\mc M \models \neg \, \varphi(\num{n_0}, \ldots, \num{n_i})$. 
\end{enumerate}

Firstly, we prove that the statement holds for atomic formulas. 
We distinguish the following cases. 
\begin{itemize}
	\item $\varphi$ is of the form $x_0 = x_1$. \\
	1. If $\N \models \num{n_0} = \num{n_1}$, then $n_0 = n_1$, and hence $\mc M \models \num{n_0} = \num{n_1}$. 

    2. If $\N \models \num{n_0} \neq \num{n_1}$, then $n_0 \neq n_1$. 
	By Lemma \ref{LemN}, $\mc M \models \num{n_0} \neq \num{n_1}$ because $n_0, n_1 \leq k$. 

	\item $\varphi$ is of the form $0 = x_0$. \\
	1. If $\N \models 0 = \num{n_0}$, then $0 = n_0$, and hence $\mc M \models 0 = \num{n_0}$. 

    2. If $\N \models 0 \neq \num{n_0}$, then $0 \neq n_0$. 
	By Lemma \ref{LemN}, $\mc M \models 0 \neq \num{n_0}$ because $n_0 \leq k$. 

	\item $\varphi$ is of the form $\s{x_0} = x_1$. \\
	1. If $\N \models \s{\num{n_0}} = \num{n_1}$, then $n_0 + 1 = n_1$, and $\mc M \models \num{n_0 + 1} = \num{n_1}$. 
	This means $\mc M \models \s{\num{n_0}} = \num{n_1}$. 

	2. If $\N \models \s{\num{n_0}} \neq \num{n_1}$, then $n_0 + 1 \neq n_1$. 
	If $k = 0$, then $n_0 = n_1 = 0$. 
	By $\Ax{5}$, we get $\mc M \models \sbr{(0 \times 0) + 0} \neq 0$. 
	By Lemma \ref{LemAM}, $\mc M \models (0 \times 0) + 0 = 0$, and hence $\mc M \models \num{1} \neq 0$. 
	This means $\mc M \models \s{\num{n_0}} \neq \num{n_1}$. 
	If $k \geq 1$, then we have $n_0 + 1 \leq k^2 + k$. 
    By Lemma \ref{LemN}, we have $\mc M \models \s{\num{n_0}} \neq \num{n_1}$. 

	\item $\varphi$ is of the form $x_0 + x_1 = x_2$. \\
	1. If $\N \models \num{n_0} + \num{n_1} = \num{n_2}$, then $n_0 + n_1 = n_2$, and $\mc M \models \num{n_0 + n_1} = \num{n_2}$. 
	By Lemma \ref{LemAM2}, $\mc M \models \num{n_0} + \num{n_1} = \num{n_2}$ because $n_0, n_1 \leq k$. 

	2. If $\N \models \num{n_0} + \num{n_1} \neq \num{n_2}$, then $n_0 + n_1 \neq n_2$. 
	Since $n_0 + n_1 \leq k^2 + k$ and $n_2 \leq k$, by Lemma \ref{LemN}, we have $\mc M \models \num{n_0 + n_1} \neq \num{n_2}$. 
	By Lemma \ref{LemAM2}, $\mc M \models \num{n_0} + \num{n_1} \neq \num{n_2}$. 
	
	\item $\varphi$ is of the form $x_0 \times x_1 = x_2$. \\
	1. If $\N \models \num{n_0} \times \num{n_1} = \num{n_2}$, then $n_0 \times n_1 = n_2$, and $\mc M \models \num{n_0 \times n_1} = \num{n_2}$. 
	By Lemma \ref{LemAM2}, $\mc M \models \num{n_0} \times \num{n_1} = \num{n_2}$ because $n_0, n_1 \leq k$. 

	2. If $\N \models \num{n_0} \times \num{n_1} \neq \num{n_2}$, then $n_0 \times n_1 \neq n_2$. 
	By Lemma \ref{LemN}, $\mc M \models \num{n_0 \times n_1} \neq \num{n_2}$ because $n_0 \times n_1 \leq k^2 + k$ and $n_2 \leq k$. 
	By Lemma \ref{LemAM2}, $\mc M \models \num{n_0} \times \num{n_1} \neq \num{n_2}$. 
	
	\item $\varphi$ is of the form $x_0 \leq x_1$. \\
	1. If $\N \models \num{n_0} \leq \num{n_1}$, then $n_0 \leq n_1$. 
	Since $\mc M \models \bigvee_{l \leq n_1} \num{n_0} = \num{l}$, we have, by Lemma \ref{LemL}, 
 $\mc M \models \num{n_0} \leq \num{n_1}$ because $n_1 \leq k$. 

	2. If $\N \models \num{n_0} \not \leq \num{n_1}$, then $n_1 < n_0$. 
	For each $l \leq n_1$, we have $n_0 \neq l$. 
	Since $n_0, l \leq k$, we have $\mc M \models \num{n_0} \neq \num{l}$ by Lemma \ref{LemN}. 
	Then, $\mc M \models \bigwedge_{l \leq n_1} \num{n_0} \neq \num{l}$. 
	Since $n_1 \leq k$, by Lemma \ref{LemL}, we have $\mc M \models \num{n_0} \not \leq \num{n_1}$. 
\end{itemize}

Secondly, we prove the induction steps. 
The case that $\varphi$ is one of the forms $\neg \, \varphi_0$ and $\varphi_0 \ast \varphi_1$ for $\ast \in \{\land, \lor, \to\}$ is easily shown by the induction hypothesis. 
It suffices to show the case that $\varphi$ is of the form $\exists y \leq x_j\, \varphi_0(x_0, \ldots, x_i, y)$, where the claim holds for $\varphi_0$. 
The case that $\varphi$ is of the form $\forall y \leq x_j\, \varphi_0(x_0, \ldots, x_i, y)$ is proved similarly. 

	1. Suppose $\N \models \exists y \bleq \num{n_j}\; \varphi_0(\num{n_0}, \ldots, \num{n_i}, y)$. 
	Then, there exists an $n_{i+1} \leq n_j \leq k$ such that $\N \models \varphi_0(\num{n_0}, \ldots, \num{n_i}, \num{n_{i+1}})$. 
	By the induction hypothesis, $\mc M \models \varphi_0(\num{n_0}, \ldots, \num{n_i}, \num{n_{i+1}})$. 
	Also, we have already proved $\mc M \models \num{n_{i+1}} \leq \num{n_j}$. 
	Therefore, we obtain $\mc M \models \exists y \bleq \num{n_j}\, \varphi_0(\num{n_0}, \ldots, \num{n_i}, y)$. 
	
	2. Suppose we have $\N \models \neg \, \exists y \bleq \num{n_j}\, \varphi_0(\num{n_0}, \ldots, \num{n_i}, y)$, or, equivalently, 
    $\N \models \forall y \bleq \num{n_j}\, \neg \, \varphi_0(\num{n_0}, \ldots, \num{n_i}, y)$. 
	Then, we have $\N \models \neg\, \varphi_0(\num{n_0}, \ldots, \num{n_i}, \num{l})$, for each $l \leq n_j$. 
	By the induction hypothesis, we have $\mc M \models \neg\, \varphi_0(\num{n_0}, \ldots, \num{n_i}, \num{l})$ for each $l \leq n_j$. 
	Thus, \[\mc M \models \forall y \Bigl(\bigvee_{l \bleq n_j} y = \num{l} \to \neg \, \varphi_0(\num{n_0}, \ldots, \num{n_i}, y) \Bigr).\] 
	By Lemma \ref{LemL}, we obtain $\mc M \models \forall y \bleq \num{n_j}\, \neg \, \varphi_0(\num{n_0}, \ldots, \num{n_i}, y)$, 
 and, hence, $\mc M \models \neg \, \exists y \bleq \num{n_j}\, \varphi_0(\num{n_0}, \ldots, \num{n_i}, y)$. 
\end{proof}

We have finished our investigation of the model $\mc M$. 
We prove one further lemma as a bridge between the conditions on $\mc M$ of the preceding lemmas and a false $\Sigma_1$-sentence
and, then, we are ready for the proof of Theorem \ref{QS}.

\begin{lem}\label{keukensmurf}
Let $\sigma$ be a pure 1-$\Sigma_1$-sentence of the form $\exists x \, \sigma_0(x)$, where $\sigma_0(x)$ is a pure $\Delta_0$-formula. 
    Suppose $\N \models \neg\,\sigma$. 
    Let $\mc N$ be an $\LA$-model and $\virtob \in \mc N$ be a witness of $\sigma^{\sq}$ in $\mc N$, that is, 
    $\mc N \models \virt(\virtob) \wedge \sigma_0(\virtob)$.
    Then $\mc N \models \num{m} < \virtob$, for every $m\in \omega$.
\end{lem}

\begin{proof}
    Suppose $\N \models \neg\,\sigma$ and 
    $\mc N \models \virt(\virtob) \wedge \sigma_0(\virtob)$.

We first show that, for every $m \in \omega$, we have $\mc N \models \num{m} \neq \virtob$. 
Suppose $\mc N \models \num{m} =\virtob$, for some $m \in \omega$. 
Let $m^\star$ be the least such $m$. 
Then, the condition (\dag) holds for $\mc N$, $\virtob$, and $m^\star$. 
Since $\N \models \neg \, \sigma_0(\num{m}^\star)$, we have, by Lemma~\ref{vrolijkesmurf},
$\mc N \models \neg \, \sigma_0(\num{m}^\star)$. 
But this is impossible.

  Since, for every $m$, we have $\mc M \models \num{m} \neq \virtob$, Lemma~\ref{Lem1} gives us that,
  for every $m$, we have $\mc M \models \num{m} < \virtob$.
\end{proof}

We  now prove Theorem \ref{QS}. 

\begin{proof}[Proof of Theorem \ref{QS}]
\emph{Ad 1:} The implication from $\sigma^{\sq}$ to $\sigma$ is immediate.

\medskip
\emph{Ad 2:} It is obvious that $\N \models \forall v\, \virt(v)$ holds. 
Suppose $\N \models \sigma$. It follows that $\N \models \sigma^{\sq}$.
Since $\sigma^{\sq}$ is $\Sigma_1$, we have, by $\Sigma_1$-completeness (Theorem \ref{SCompl}),  
that  $\RR_0 \vdash \sigma^{\sq}$. 

\medskip
\emph{Ad 3:} Suppose $\N \models \neg\,\sigma$. We prove that the theory $\sigma^{\sq}$ is $\Sigma_1$-complete. 

Let $\mc N$ be any $\LA$-model of $\sigma^{\sq}$. 
So, for some $\virtob$, we have
$\mc N \models \virt(\virtob) \wedge \sigma_0(\virtob)$. 
By Lemma~\ref{keukensmurf}, we find that $\mc N \models \num{m} \neq \virtob$ for every $m \in \omega$.
Hence, the condition (\dag) holds for $\mc N$, $\virtob$, and all $k \in \omega$.

Let $\psi$ be any $\Sigma_1$-sentence such that $\N \models \psi$. 
By Proposition~\ref{pure2}, there exists a pure $\Delta_0$-formula $\delta(x)$ satisfying the following conditions: 
\begin{enumerate}
	\item $\N \models \psi \leftrightarrow \exists x \, \delta(x)$, 
	\item $\exists x \, \delta(x) \to \psi$ is logically valid. 
\end{enumerate}
Then, $\N \models \exists x \, \delta(x)$, and, hence,
$\N \models \delta(\num{n})$ for some $n$. 
By Lemma~\ref{vrolijkesmurf}, we have $\mc N \models \delta(\num{n})$. 
By the completeness theorem, we obtain $\sigma^{\sq} \vdash \delta(\num{n})$.\footnote{In fact, the witnessing proof can be directly read off from the proofs of the lemmas.} 
Thus, $\sigma^{\sq} \vdash \exists x \, \delta(x)$. 
Therefore, we obtain $\sigma^{\sq} \vdash \psi$. 

Finally, by Theorem \ref{SCompl}, we conclude that $\sigma^{\sq} \vdash \RR_0$. 
\end{proof}

\section{\cert\ Generalized}\label{hoempapasmurf}
In this section, we generalize Theorem \ref{CSS} to a wide class of further base theories.

\subsection{The Theory $\RP$}
We start with reproving Theorem~\ref{CSS} for a slightly different base theory. We define:
\begin{itemize}
    \item 
    $\LAP$ is the arithmetical signature $\LA$ extended by a unary predicate symbol {\sf P}.
    \item
    $\idap:= \idax{\LAP}$. (The notion $\idax{}$ is explained in Definition~\ref{iddef}.)
    \item
    $\RP$ is the the $\LAP$-theory obtained by extending $\RR_0$ with the  axioms
    ${\sf P}(\num{n})$, for all $n\in \omega$.
\end{itemize}

We have:
\begin{thm}[Second Certified Extension Theorem]\label{CSS0}
For each $\Sigma_1$-sentence $\sigma$, we can effectively find a sentence $\certt\sigma$ satisfying the following conditions:
\begin{enumerate}[1.]
    \item  $\certt\sigma \vdash  \sigma$. 
    \item If $\N \models \sigma$, then $\RP \vdash \certt\sigma$. 
    \item If $\N \models \neg \, \sigma$, then $\certt\sigma \vdash \RP$. 
\end{enumerate}
\end{thm}

\begin{proof}
For each pure 1-$\Sigma_1$-sentence $\sigma$, let $\sigma^{\sf cert_p}$ be the $\LA$-sentence:
\[
	\bigwedge \mathfrak {id}_{\sf ap} \wedge
 \exists v\, \bigl(\virt(v) \land 
  \sigma_0(v) \land \forall x \bleq v\, {\sf P}(x) \bigr). 
\]
The proof of Theorem \ref{QS} can be repeated for $\RP$ by using the sentence $\sigma^{\sf cert_p}$. 
Thus, for each $\Sigma_1$-sentence $\sigma$, it is shown that $\certt\sigma = (\sigma^\grullet)^{\sf cert_p}$ satisfies the required conditions. 
\end{proof}

The addition of $\bigwedge \mathfrak {id}_{\sf ap}$ in the definition of $\certt{\sigma}$ is superfluous in the context
of the proof of Theorem~\ref{CSS0}. It is added since it also delivers the following simple insight.

\begin{thm}\label{observation}
    Suppose $\tau$ is a translation from $\LAP$ to a signature of some theory. Let $\sigma$ be a $\Sigma_1$-sentence.  Then,
    ${\sf K}_\tau:\certt{\sigma}^\tau \rhd \certt{\sigma}$, where ${\sf K}_\tau$ is the interpretation based on $\tau$.
\end{thm}

\subsection{$\RP$-sourced Theories}

We define:
\begin{itemize}
    \item
    Let $T$ and $U$ be c.e.~theories and let $\Theta$ be the signature of $T$. 
    Let $\tau$ be a translation from $\Theta$ to the $U$-signature. 
    The theory $U$ is \emph{$\tau$-$T$-sourced} iff $U$ is deductively equivalent to $(T+ \idax{\Theta})^{\tau}$.
    The theory $U$ is $T$-sourced if it is $\tau$-$T$-sourced, for some $\tau$.
\end{itemize}

In this paper we just focus on $\RP$-sourced theories.
We have:

\begin{thm}[Generalized Certified Extension Theorem]\label{CSS1}
Suppose $W$ is $\tau$-\gener.
Then, for each $\Sigma_1$-sentence $\sigma$, we can effectively find a sentence $\certt{\sigma}$ satisfying the following conditions: 
\begin{enumerate}[1.]
    \item  $\certt\sigma^\tau \vdash  \sigma^\tau$. 
    \item If $\N \models \sigma$, then $W \vdash \certt\sigma^\tau$. 
    \item If $\N \models \neg \, \sigma$, then $\certt\sigma^\tau \vdash W$. 
\end{enumerate}
\end{thm}

\begin{proof}
This is immediate from Theorem~\ref{CSS0}. Note that, for (1), we use Theorem~\ref{observation}.
\end{proof}

Here is a first simple example of a \gener\ theory.
\begin{ex}
    Any finitely axiomatized theory $A$ that interprets $\RR_0$ is \gener.
    We note that this example is not very useful, since we already know that the applications we
    want from \gener\ theories hold for finitely axiomatized theories.
\end{ex}

Here is a second example. 
We remind the reader of Vaught's set theory {\sf VS}.
It is a theory in the signature with the single binary relation symbol $\in$, axiomatized by the following axioms.
\begin{description}
    \item[${\sf VS}n$.]
   $ \forall x_0 \dots \forall x_{n-1}\, \exists z\, \forall y \, (y \in z \iff \bigvee_{i<n}y =x_i)$
\end{description}

\noindent
We note that in the case that $n=0$, we have an axiom that guarantees the existence of some empty sets.
We have:

\begin{thm}
    {\sf VS} interprets $\RR_0$.
\end{thm}

\begin{proof}
    In \cite[Appendix A]{viss:enay21}, it was proven that {\sf VS} interprets $\RR$.
    So, \emph{a fortiori}, {\sf VS} interprets $\RR_0$.
\end{proof}

\begin{thm}\label{VS}
    The theory {\sf VS} is $\RR_{0{\sf p}}$-sourced.
\end{thm}

The idea of the proof is to represent {\sf VS} as Adjunctive Set Theory {\sf AS} with local size restrictions on the sets
to which one can apply adjunction.

\begin{proof}
Let $\nu$ be the translation on which an interpretation of $\RR_0$ in {\sf VS} is based.
We assume that $\nu$ is one-dimensional---as the translation provided by \cite{viss:enay21} is. 
The many-dimensional case only requires minor adaptations.

We extend $\nu$ to $\nu^\star$ by providing a translation of {\sf P}. 
The statement ${\sf P}_{\nu^\star}( x)$ will roughly say that ${\sf VS}0$ and ${\sf VS}2$ and that, whenever
we have a set $y$ of cardinality $\leq x$, we may adjoin any $z$ to $y$. Here ${\sf VS}0$ and ${\sf VS}2$ are needed to provide the
necessary coding machinery. Here are the ingredients for formulation of our statement.
\begin{itemize}
    \item 
    $f$ is an injection from objects to numbers or
    ${\sf in}(f)$ iff $f$ is a set of pairs, $f$ is functional w.r.t.~$=$, the objects in the
    range of $f$ are in $\delta_\nu$, and
    $f$ is injective in the sense that, if $f(u) =_\nu f(v)$, then $u=v$.
    \item 
    ${\sf dom}(f,y)$ iff $\forall w\, (w\in y \iff \exists v\, f(w)=v)$.
    \item 
    ${\sf card}_{\leq} (y,x)$ iff there is an $f$ with ${\sf in}(f)$ and ${\sf dom}(f,y)$ and 
    $\forall v\in y\, f(v) <_\nu x $.
    \item 
    ${\sf P}_{\nu^\star}(x)$ iff ${\sf VS}0$ and ${\sf VS}2$ and 
    \[\forall y\, \forall z\, ({\sf card}_{\leq}(y,x) \to \exists w \, \forall u\, (u \in w \iff (u\in x \vee u =z))).\] 
\end{itemize}
The deductive equivalence between $\RR^{\nu^\star}_{0{\sf p}}$ and {\sf VS} is now easy to verify.
\end{proof}

We provide two larger classes of
examples of \gener\ theories. Consider any signature $\mc L$. Let $\tau_0$ be a translation of $\mc L$ in $\LA$.

Let $\chi_0(x_0, \ldots, x_{k_0-1})$, \ldots, $\chi_m(x_0, \ldots, x_{k_m-1})$ be 
any $\mc L$-formulas and let
$X_0, \ldots, X_m$ be any computable relations on $\omega$ such that $X_i \subseteq \omega^{k_i}$ for each $i \leq m$.
We write $\num{n}$ for the $\tau_0$-numerals in the context of $\mc L$.
We define the theory $\RR_0[\tau_0;\chi_0, \ldots, \chi_m; X_0, \ldots, X_m]$ as follows:
\[
	\RR_0^{\tau_0} + \bigcup_{i \bleq m} \{\chi_i(\num{n_0}, \ldots, \num{n_{k_i-1}}) \mid (n_0, \ldots, n_{k_i-1}) \in X_i\}. 
\]
We use $\RR_0[\tau_0;\vv{\chi}; \vv{X}]$ as an abbreviation of
$\RR_0[\tau_0;\chi_0, \ldots, \chi_m; X_0, \ldots, X_m]$. If $\mc L$ is $\LA$ and if $\tau_0$ is the identity translation on
$\LA$, we simply omit $\tau_0$.

\begin{thm}
Let $\mc L$ be a signature, let $\tau_0$ be a translation of $\LA$ into $\mc L$.  Let $\vv \chi$ be 
any $\mc L$-formulas and let
$\vv X$ be computable relations on $\omega$ matching $\vv \chi$.
Then, $\RR_0[\tau_0;\vv{\chi}; \vv{X}]$ is \gener.
\end{thm}

\begin{proof}
To simplify inessentially we assume that $\tau_0$ is 1-dimensional.  
For each $i \leq m$, let $\xi_i(x_0, \ldots, x_{k_i-1})$ be an $\LA$-formula representing $X_i$ in $\RR_0$. 
That is, for any $n_0, \ldots, n_{k_i-1} \in \omega$, if $(n_0, \ldots, n_{k_i-1}) \in X_i$, then 
$\RR_0 \vdash \xi_i(\num{n_0}, \ldots, \num{n_{k_i-1}})$; and if $(n_0, \ldots, n_{k_i-1}) \notin X_i$, then 
$\RR_0 \vdash \neg\, \xi_i(\num{n_0}, \ldots, \num{n_{k_i-1}})$. 
Let $\chi^*(x)$ be the $\mc L$-formula
\begin{align*}
	\bigwedge_{i \bleq m} & 
 \forall x_0 \bleq_{\tau_0} x\, \cdots \forall x_{k_i-1} \leq_{\tau_0} x\, 
 (\xi^{\tau_0}_i(x_0, \ldots, x_{k_i-1}) \to \chi_i(x_0, \ldots, x_{k_i-1})).
\end{align*}
We extend $\tau_0$ to $\tau$ by setting ${\sf P}_\tau(x) := \chi^\ast(x)$.
Then, clearly, $\RR_0[\tau_0; \vv{\chi}; \vv{X}]$ is deductively equivalent to $\RP^\tau$. 
\end{proof}

\begin{ex}
The theory $\RR_0$ is \gener\ since it is deductively equivalent to
$\RR_0[x=x; \emptyset]$ and the theory
$\RR$ is \gener\ since it is deductively equivalent to 
$\RR_0[\forall y\, (y \leq x \lor x \leq y); \omega]$. 
\end{ex}

\begin{ex}\label{leergierigesmurf}
    Suppose $\forall \vv x\, P_0(\vv x)$ is a true pure $\Pi_1$-sentence. Consider $W :=\RR_0[P_0;\omega^n]$, where $n$ is the length of $\vv x$.
    Then, by $\Sigma_1$-completeness, $W$ is deductively equivalent to $\RR_0$. However, modulo deductive equivalence,
    the corresponding $(\sigma^{\sf cert_p})^\tau$ has the form 
    $\exists z\,  (\virt(z) \wedge \forall \vv x \bleq z\, P_0(\vv x) \wedge \exists z\, \sigma_0(z))$.
    So, we may think of $\virt(z) \wedge \forall \vv x \bleq z\, P_0(\vv x)$ as replacing $\virt(z)$. Thus, we may add all kind of
    desirable properties to $\virt$ like the linearity of $\leq$ below the certified element. This may be useful if we want to use
    \vness\ as the basis of an interpretation of a stronger theory in $\RR_0$.
\end{ex}

In the development below, we use the notion of \emph{depth-of-quantifier-alternations complexity}, henceforth, simply \emph{complexity}.
For a careful exposition of this notion, see \cite{viss:smal19}.

\begin{thm}\label{restrictionthm}
Consider a finitely axiomatized theory $A$ and a number $n$ and a computable set $X$ of $A$-sentences of
complexity $\leq n$. Suppose that:
\begin{itemize}
    \item 
    $A$ interprets $\RR_0$ via an interpretation based on translation $\tau_0$.
    \item
    There is an $A$ formula $\varPhi$ such that $A \vdash  \Phi(\gn{\phi}) \iff \phi$, for all
    $A$-sentences $\phi$ of complexity $\leq n$. Here the numerals are the $\tau_0$-numerals. \textup(Note that these
    could be sequences modulo a definable equivalence relation.\textup)
\end{itemize}
Then $A +  X$ is an \gener\ c.e.~theory.
\end{thm}

\begin{proof}
We write $\top$ for the non-empty zero-ary relation\footnote{If we represent relations as sets of tuples, then
$\top = \verz{\varepsilon}$, where $\varepsilon$ is the empty sequence.} and $\bigwedge A$ for
the sentence that is the conjunction of the axioms of $A$.
    We note that $A+X$ is deductively equivalent to $\RR_0[\tau_0;\bigwedge A,\varPhi; \top,X]$.
\end{proof}

A theory is \emph{restricted}, iff it can be axiomatized by axioms of complexity $\leq n$, for some fixed $n$. 
In \cite{viss:smal19} it has been verified that sequential restricted c.e.~theories can be written
as $A+X$, where $A$ and $X$ satisfy the conditions of Theorem~\ref{restrictionthm}. So, we have:

\begin{cor}\label{smartsmurf}
    Any restricted sequential c.e.~theory is \gener.
\end{cor}

Examples of restricted sequential c.e.~theories are:
\begin{itemize}
\item
$\mho_{\sf PA} :={\sf S}^1_2+ {\sf Con}_1({\sf PA}) + {\sf Con}_2({\sf PA}) + \dots$, where the ${\sf Con}_i(\PA)$ are
consistency statements where we restrict the {\sf PA}-axioms to those with G\"odel number $\leq i$ and where we restrict
the proofs to those in which only formulas of depth of quantifier alternations $\leq i$ occur.
    \item 
$\mathrm{I}\Delta_0+\Omega_1+ \Omega_2+\dots$,
\item
${\sf EA}+{\sf Con}({\sf EA})+
{\sf Con}({\sf EA}+{\sf Con}({\sf EA}))+ \dots$,
\item
{\sf PRA} (in a suitable version in finite
signature).
\end{itemize}

We end with a closure condition.

\begin{thm}
    Suppose $U$ and $V$ are \gener\ theories in the same signature $\mc L$, as witnessed by $\tau$ and $\tau'$.
    Suppose the restriction of $\tau$ and $\tau'$ to the arithmetical signature is a shared part $\tau_0$.
    Then $U \cup V$ is \gener.
\end{thm}

\begin{proof}
    We take as witnessing translation $\tau^\star$ for $U\cup V$, the translation $\tau_0$ on the arithmetical vocabulary and
    ${\sf P}_{\tau^\star}(\vv x) := ({\sf P}_\tau(\vv x) \wedge {\sf P}_{\tau'}(\vv x))$. Alternatively, we note that
    $U \cup V$ is deductively equivalent to $\RR_0[\tau_0; {\sf P}_\tau(\vv x),{\sf P}_{\tau'}(\vv x);\omega,\omega]$.
\end{proof}

\section{Witness Comparisons and Fixed Points}\label{wicomp}
In this section, we give the basic definitions and facts for witness comparison. Moreover, we discuss the
G\"odel Fixed Point Lemma and its interaction with witness comparison.

\subsection{Comparing the Witnesses}
We define witness comparison.
Let $\phi := \exists x\, \phi_0(x)$ and $\psi := \exists y\, \psi_0(y)$. We define:
\begin{itemize}
    \item $\phi \bleq \psi :\iff \exists x\, (\phi_0(x) \wedge \forall y < x\, \neg\,\psi_0(y))$,
    \item $\phi < \psi :\iff \exists x\, (\phi_0(x) \wedge \forall y \bleq x\, \neg\,\psi_0(y))$,
    \item $(\phi \bleq \psi)^\bot := (\psi<\phi)$,
    \item $(\phi < \psi)^\bot := (\psi\bleq \phi)$.
\end{itemize}

We have to do some preliminary work to compensate for the fact that in $\RR_0$ we are lacking  the axiom {\sf R}5 which says, for every numeral $n$, that
$x \leq \num n \vee\num n\leq x$. We say that $x$ is \emph{well-behaved} or ${\sf wb}(x)$ iff it satisfies {\sf A}1 and {\sf A}2 of the definition of certification, i.o.w.,
${\sf wb}(x)$ iff $0 \leq w$ and $\forall  y \bles x\; \s y \leq x$. We say that a sentence is \emph{well-behaved} if it is of the form $\exists x\, ({\sf wb}(x) \wedge \psi(x))$.

\begin{rem}
The idea of well-behavedness, though not the name, is due to Cobham, see
Jones and Shepherdson \cite{JS83}. 

In fact, for our purposes, we could also have worked with \emph{certified} in stead of
\emph{well-behaved}, but we found it attractive to use the lightest possible means to obtain the results.
\end{rem}
We have the following lemma.
\begin{lem}\label{heelkleinesmurf}
\begin{enumerate}[a.]
\item
$\RR_0 \vdash {\sf wb}(x) \to (x< \num n \vee \num n \leq x)$.
\item
$\RR_0 \vdash {\sf wb}(x) \to (x\leq \num n \vee \num n < x)$.
\end{enumerate}
\end{lem}

\begin{proof}
We prove (a) by a simple induction on $n$ and (b) is immediate from (a).
\end{proof}
It follows that:
\begin{lem}\label{hulpsmurf}
 Suppose $\sigma$ and $\sigma'$ are 1-$\Sigma_1$-sentences and $\sigma'$ is well-behaved.
  \begin{enumerate}[a.]
        \item 
        If  $\N \models \sigma\bleq \sigma'$, then $\RR_0 \vdash  \neg\, \sigma'\bles\sigma$. 
                       \item 
       If  $\N \models \sigma\bles \sigma'$, then, $\RR_0 \vdash  \neg\, \sigma'\bleq\sigma$.
    \end{enumerate}
\end{lem}

\begin{proof}
Let $\sigma = \exists x\,\sigma_0(x)$ and $\sigma' = \exists y\,({\sf wb}(y) \wedge \sigma_0'(y))$.

\emph{We treat \textup(a\textup).}
Suppose $\sigma\bleq \sigma'$ is true in the natural numbers. Then, for some $n$, we have
$\RR_0 \vdash \sigma_0(\num n)$. Moreover, for all $k< n$, we have (\dag) $\RR_0\vdash \neg\,\sigma_0'(\num k)$. 
We reason in $\RR_0$. Suppose $ \sigma'\bles\sigma$. 
Then, for some well-behaved $y$, we have $\sigma_0'(y)$ and (\ddag) $\forall z\bleq y\, \neg\, \sigma_0(z)$.
By Lemma~\ref{heelkleinesmurf}, we find that $y < \num n$ or $\num n \leq y$. 
The first disjunct contradicts (\dag) and the second disjunct contradicts (\ddag).

The proof of (b) is similar.
\end{proof}

We prove the result that gives us the desired applications.
\begin{thm}\label{ijverigesmurf}
     \begin{enumerate}[a.]
     Suppose $\sigma$ and $\sigma'$ are 1-$\Sigma_1$-sentences, $\sigma'$ is well-behaved and $\rho$ is a $\Sigma_1$-sentence. 
        \item 
        If $\N \models \sigma\bleq \sigma'$ and $\RR_0 \vdash \rho \leftrightarrow \sigma'<\sigma$, then $[\rho]$ is inconsistent.
        \item 
        If $\N \models \sigma< \sigma'$ and $\RR_0 \vdash \rho \leftrightarrow \sigma'\bleq \sigma$, then $[\rho]$ is inconsistent.
    \end{enumerate}
\end{thm}

\begin{proof}
    %
   We  prove (a). Suppose $\N \models \sigma\bleq \sigma'$ and
    $\RR_0 \vdash \rho \leftrightarrow \sigma'\bles\sigma$. 
    By Lemma~\ref{hulpsmurf}, we have
    $\RR_0 \vdash \neg\, \sigma'\bles \sigma$, and hence $\RR_0 \vdash \neg \, \rho$. 
    It follows that $\N \models \neg\, \rho$. So, by Theorem~\ref{CSS}, that 
    $[\rho] \vdash \RR_0$, and, hence, $[\rho] \vdash \neg\, \rho$.
    On the other hand, again by Theorem~\ref{CSS}, we have $[\rho] \vdash \rho$.
    \emph{Ergo}, $[\rho]$ is inconsistent.
\end{proof}

By a trivial adaptation of the above argument, we also have:
\begin{thm}\label{ijverigesmurf2}
Let $\sigma$ and $\sigma'$ be 1-$\Sigma_1$-sentences, where $\sigma'$ is well behaved. Let $\rho$ be any $\Sigma_1$-sentence. We have:
    \begin{enumerate}[a.]
        \item 
        If $\N \models \sigma\bleq \sigma'$ and $\RR_0 \vdash \rho \leftrightarrow \sigma'\bles \sigma$, then $\certt{\rho}$ is inconsistent.
        \item 
        If $\N \models \sigma \bles \sigma'$ and $\RR_0 \vdash \rho \leftrightarrow \sigma'\bleq \sigma$, then $\certt{\rho}$ is inconsistent.
    \end{enumerate}
\end{thm}

\subsection{The Fixed Point Lemma}
In $\RR_0$ we can prove the representability of all recursive functions as the following lemma describes.

\begin{lem}\label{smurferella2}
For every recursive function $F$ there is a 1-$\Sigma_1$-formula $\sigma(x,y)$ such that, whenever
$F(n)=m$, we have $\RR_0 \vdash \forall y\, (\sigma(\num n,y) \iff \num m=y)$.
\end{lem}

\begin{proof}
Consider any recursive function $F$ and let $\sigma^\star(x,y) = \exists z\, \sigma^\star_0(x,y,z)$ be any 1-$\Sigma_1$-formula representing the
graph of $F$. 
We take
 \begin{multline*}
 \sigma(x,y) :\iff  \exists z\, \bigl({\sf wb}(z) \wedge y \leq z \wedge \exists u \leq z\, \sigma^\star_0(x,y,u) \,
\wedge \\ \forall a \leq z\, \forall b \leq z\, (\sigma^\star_0(x,a,b) \to a = y)\bigr).
\end{multline*}
We now use Lemma~\ref{heelkleinesmurf}, to mimick the well-known proof of the analogue of the Lemma for the case of $\RR$. 
\end{proof}

We can  prove the usual fixed point lemma using a representation of the substitution function provided by
Lemma~\ref{smurferella2}. However, we need a bit more.

\begin{thm}\label{puntneussmurf}
\begin{enumerate}[i.]
    \item Suppose $\sigma(x)$ is $\Sigma_1$. Then, we can find a $\Sigma_1$-formula
$\eta$ such that $\RR_0 \vdash \eta \iff \sigma(\gn \eta)$.
\item 
Suppose $\sigma(x,y)$ and $\sigma'(x,y)$ are  $\Sigma_1$-formulas. We can find $\Sigma_1$-formulas
$\eta$ and $\eta'$ such that $\RR_0 \vdash \eta \iff \sigma(\gn \eta,\gn {\eta'})$ and
$\RR_0 \vdash \eta' \iff \sigma'(\gn \eta,\gn {\eta'})$.
\end{enumerate}
\end{thm}

\begin{proof}
\emph{We treat \textup(i\textup).}

We can obtain the desired result by a careful modification of the usual proof of the Fixed Point Lemma. Alternatively, we can proceed as follows. 
Let $\Sigma_1^\dag$ be the class given by $\chi::= \sigma \mid (\chi \wedge \chi) \mid \exists v\,\chi$. Here $\sigma$ ranges over $\Sigma_1$-sentences.
We can easily rewrite a $\Sigma_1^\dag$-sentence to a $\Sigma_1$-sentence by moving the relevant  existential quantifiers out. 
The usual fixed point calculation delivers a  $\Sigma^\dag_1$-sentence using the wide scope elimination for the substitution function.
Suppose we arithmetize this normalization function say as {\sf norm}. We can represent this function in $\RR_0$ by Lemma~\ref{smurferella2}.
We can construct  a $\Sigma_1$-formula $\sigma'(x)$ that functions as $\sigma({\sf norm}(x))$. Let $\eta'$ be the ordinary G\"odel fixed point of
$\sigma'$ and let $\eta$ be the normalized form of $\eta'$. Then, we have:
\begin{eqnarray*}
    \RR_0 \vdash \eta & \iff & \eta' \\
    & \iff & \sigma'(\gn{\eta'}) \\
    & \iff & \sigma(\gn{\eta}). 
\end{eqnarray*}

The alternative proof is easily adapted to deliver also the desired double fixed point promised in (ii).
\end{proof}

We note that the trick of the alternative proof will work for all kinds of normalizations.

A disadvantage of the modified G\"odel-style fixed point construction is that it does not preserve witness comparison form.
For the results in the next sections, this is not really needed, but it may sometimes lead to less elegant formulations.
How nice it would be if we could preserve almost all forms of the original formula.
An elegant way to do this is to employ a self-referential G\"odel-numbering, which has self-reference built in.

The idea of a self-referential G\"odel numbering was introduced by
Saul Kripke in \cite[Footnote 6]{krip:outl75}. It was worked out in some detail in \cite{viss:sema04}. 
Recently, two papers appeared exploring the notion further, to wit
\cite{krip:gode23} and \cite{grab:self23}. As far as we know the only place where the idea
is truly applied is \cite{grab:inva21}.

We  follow the realization of \cite{viss:sema04}. The idea is simple. We extend $\LA$ with a fresh constant
{\sf c} and employ a standard G\"odel numbering for the extended language. We then take the standard G\"odel number $\gn{\phi({\sf c})}$ for the extended language to be
the self-referential G\"odel number of $\phi(\gn{\phi({\sf c})})$ for $\LA$. Without further measures, the resulting G\"odel numbering is not 
functional from sentences to numbers. The functionality of a G\"odel number is not strictly needed, but we opt for 
 making the numbering functional by stipulating that we take the smallest number assigned to
a sentence by our non-functional version. In  \cite{viss:sema04} it is carefully verified that, for a decent choice of the
input standard G\"odel numbering, the numbering so obtained fulfills all the desiderata of a self-referential G\"odel numbering.
We  write $\sgn{\phi}$ for the self-referential G\"odel number of $\phi$. 
The crucial property is that for any $\phi(x)$ with at most $x$ free, we can effectively find a $\psi$ with
$\psi = \phi(\sgn \psi)$.

In this paper we mostly opt for the ordinary numbering, accepting the use of Theorem~\ref{puntneussmurf} as the way to go.
We will prove Theorem~\ref{SFF} twice, once with an
ordinary numbering and once with a self-referential one in order to illustrate the use of a self-referential numbering.

\section{Vaught's Theorems Revisited}\label{vaughtagain}
We give two  proofs of Theorem~\ref{Vau1} and prove a generalization of
Theorem~\ref{Vau2}. 

\newtheorem*{Vaught1}{Theorem \ref{Vau1}}
\newtheorem*{Vaught2}{Theorem \ref{Vau2}}

\begin{Vaught1}[Vaught {\cite[5.2]{Vau62}}]
The theory $\RR_0$ is strongly effectively inseparable. 
\end{Vaught1}

\begin{proof}[First Proof.]
Let $i, j \in \omega$ be such that $\RR_{0\mathfrak{p}} \subseteq \W{i}$, $\emptyset_{\mathfrak{r}} \subseteq \W{j}$, and $\W{i} \cap \W{j} = \emptyset$. 
Let $(X, Y)$ be any effectively inseparable pair of c.e.~sets. We can clearly find 1-$\Sigma_1$-formulas  $\xi(x)$ and $\eta(x)$ that represent
$X$ respectively $Y$. We can clearly arrange that $\xi$ and $\eta$ are well-behaved.

For each natural number $n$, let $\sigma_n$ be the 1-$\Sigma_1$-sentence  $\xi(\num n) \bleq \eta(\num n)$.
We can effectively find natural numbers $i'$ and $j'$ such that we have $\W{i'} = \{n \in \omega \mid [\sigma_n] \in \W{i}\}$ 
and $\W{j'} = \{n \in \omega \mid [\sigma_n]\in \W{j}\}$.\footnote{In order to avoid notational overload, we omit the G\"odel numbering brackets
around $\sigma_n$.}
Obviously, $\W{i'} \cap \W{j'} = \emptyset$. 

Suppose $n \in X$. 
Since $X$ and $Y$ are disjoint, we find $\N \models \sigma_n$. 
By Theorem \ref{CSS}, we have $\RR_0 \vdash [\sigma_n]$,
that is, $[\sigma_n] \in \RR_{0\mathfrak{p}}$. 
Then, $[\sigma_n] \in \W{i}$, and hence $n \in \W{i'}$. 

Suppose $n \in Y$. We find $\N \models \sigma_n^\bot$.
By Theorem~\ref{ijverigesmurf}, we obtain $ [\sigma_n] \vdash \bot$, that is
$[\sigma_n] \in \emptyset_{\mathfrak{r}}$.
Hence, $[\sigma_n] \in \W{j}$, and, so, $n \in \W{j'}$. 

We have shown that $\W{i'}$ and $\W{j'}$ separate $(X, Y)$. 
By the effective inseparability of $(X, Y)$, we can effectively find a number $m^\star$ such that $m^\star \notin \W{i'} \cup \W{j'}$. 
Then, $[\sigma_{m^\star}] \notin \W{i} \cup \W{j}$. 
Thus, we have shown that $(\RR_{0\mathfrak{p}}, \emptyset_{\mathfrak{r}})$ is effectively inseparable. 
\end{proof}

\begin{proof}[Second Proof.]
Let $X$ and $Y$ be any c.e.~sets separating $\RR_{0\mf p}$ and $\emptyset_{\mf r}$.
We assume that $x \in X$ and $x\in Y$ are represented by well-behaved 1-$\Sigma_1$-formulas $\xi$ and $\eta$.
By  Theorem~\ref{puntneussmurf}, we can effectively find a $\Sigma_1$-sentence $\rho$ satisfying the following equivalence: 
$ \RR_0 \vdash \rho \leftrightarrow \eta(\certo{\rho}) \bleq \xi(\certo{\rho}) $.

Suppose $\certo{\rho} \in X$. 
Since $X$ and $Y$ are disjoint, we have $\certo{\rho} \notin Y$. 
It follows that $\mathbb N \models \xi(\certo{\rho}) \bles \eta(\certo{\rho})$ and, thus, by Theorem~\ref{ijverigesmurf},
that $\certo{\rho}$ is inconsistent, i.e., $\certo{\rho} \in \emptyset_{\mathfrak{r}}$. 
We may conclude that $X \cap \emptyset_{\mathfrak{r}}$ is non-empty. But this is a contradiction. 

Suppose $\certo{\rho} \in Y$. 
Since $X$ and $Y$ are disjoint, we have $\N \models \rho$. 
By Theorem \ref{CSS}, we find $\RR_0 \vdash \certo{\rho}$, i.e.,
$\certo{\rho} \in \RR_{0\mathfrak{p}}$. 
Thus, $Y \cap \RR_{0\mathfrak{p}}$ is non-empty. 
This is a contradiction. 

Therefore, $\certo{\rho} \notin X \cup Y$. 
Since we can find $\certo{\rho}$ effectively, we have shown that 
$(\RR_{0\mathfrak{p}}, \emptyset_{\mathfrak{r}})$ is effectively inseparable. 
\end{proof}

Let $\mc{F}$ be the set of all $\LA$-sentences having a finite model. 
Since $\RR_{0\mathfrak{p}} \subseteq \mc{F}$ and $\mathcal{F} \cap \emptyset_{\mathfrak{r}} = \emptyset$, 
we obtain the following version of Trakhtenbrot's theorem as a corollary. 

\begin{cor}[Trakhtenbrot \cite{Tra53}]
The pair $(\mathcal{F}, \emptyset_{\mathfrak{r}})$ is effectively inseparable. 
\end{cor}

We note that our version of Trakhtenbrot's Theorem is formulated for the signature $\LA$. We can generalize it to other
signatures (with at least one relation symbol with arity $\geq 2$) by the usual tricks to translate a signature of finite
signature into the signature with one binary relation symbol; see e.g. \cite[Chapter 5.5]{hodg:mode93}.
Alternatively, we can replace the use of theories-of-a-number by the use of very weak set theories as developed, e.g., in
\cite{pakh:weak19}.

We then turn to Theorem~\ref{Vau2}. 
Before generalizing Theorem~\ref{Vau2}, we introduce the following notions. 

\begin{defn}
Let $T$ be an $\mc L$-theory and $\mc X$ be a set of $\mc L$-sentences. 
\begin{enumerate}
    \item We say that $T$ is \emph{effectively half-essentially $\mc X$-incomplete} iff there exists a partial computable function $\Phi$ such that for any natural number $i$, if $\W{i}$ is a c.e.~$\mc L$-theory such that $T + \W{i}$ is consistent, then $\Phi(i)$ converges, $\Phi(i) \in \mc X$, $T \nvdash \Phi(i)$ and $\W{i} \nvdash \neg\, \Phi(i)$. 

    \item We say that $T$ is \emph{$\mc X$-creative} iff there exists a partial computable function $\Psi$ such that for any natural number $i$, if $T_{\mathfrak{p}} \cap \W{i} = \emptyset$, then $\Psi(i)$ converges, $\Psi(i) \in \mc X$, and $\Psi(i) \notin T_{\mathfrak{p}} \cup \W{i}$. 
\end{enumerate}
\end{defn}

Actually, we prove that these two notions are equivalent. 

\begin{prop}\label{X_creativity}
For any $\mc L$-theory $T$ and any set $\mc X$ of $\mc L$-sentences, the following are equivalent: 
\begin{enumerate}[a.]
    \item $T$ is effectively half-essentially $\mc X$-incomplete. 
    \item $T$ is $\mc X$-creative. 
\end{enumerate}
\end{prop}
\begin{proof}
``(a) to (b)''. 
Let $\Phi$ be a partial computable function witnessing the effective half-essential $\mc X$-incompleteness of $T$. 
Let $\W{i}$ be any c.e.~set such that $T_{\mathfrak{p}} \cap \W{i} = \emptyset$. 
By using the recursion theorem, we can effectively find a natural number $k$ from $i$ such that
\[
    \W{k} = \begin{cases} \{\neg\, \Phi(k)\} & \text{if}\ \Phi(k) {\downarrow}\ \text{and}\ \Phi(k) \in \W{i}, \\
    \emptyset & \text{otherwise.} \end{cases}
\]
If $T + \W{k}$ were inconsistent, then there would be a sentence $\varphi$ such that $T \vdash \varphi$ and $\W{k} \vdash \neg\, \varphi$. 
Since $T$ is consistent, we would have $\Phi(k) {\downarrow} \in \W{i}$ and $\W{k} = \{\neg\, \Phi(k)\}$. 
In this case, we would obtain $T \vdash \Phi(k)$. 
Hence $\Phi(k) \in T_{\mathfrak{p}} \cap \W{i}$, a contradiction. 

Thus, $T + \W{k}$ is consistent. 
By the effective half-essential $\mc X$-incompleteness of $T$, we have $\Phi(k){\downarrow} \in \mc X$, $T \nvdash \Phi(k)$, and $\W{k} \nvdash \neg\, \Phi(k)$. 
In particular, $\W{k} \nvdash \neg\, \Phi(k)$ implies $\Phi(k) \notin \W{i}$. 
Therefore the partial computable function $\Psi(i) : = \Phi(k)$ witnesses the $\mc X$-creativity of $T$. 

\medskip
``(b) to (a)''.
Let $\Psi$ be a partial computable function witnessing the $\mc X$-creativity of $T$. 
Let $\W{i}$ be any c.e.~$\mc L$-theory such that $T + \W{i}$ is consistent. 
We can effectively find a number $k$ from $i$ such that $\W{k} = \W{i \mathfrak{r}}$. 
Then, $T_{\mathfrak{p}} \cap \W{k} = \emptyset$. 
By the $\mc X$-creativity of $T$, we have $\Psi(k){\downarrow} \in \mc X$ and $\Psi(k) \notin T_{\mathfrak{p}} \cup \W{i\mathfrak{r}}$. 
Therefore, the partial computable function $\Phi(i) : = \Psi(k)$ witnesses the effective half-essential $\mc X$-incompleteness of $T$. 
\end{proof}

We proceed with a generalization of Theorem~\ref{Vau2}. 
For each $\mc L$-theory $T$, let $\cothe T : = \{\varphi \mid \varphi$ is an $\mc L$-sentence and $\varphi \vdash T\}$.

\begin{thm}\label{gen_Vaught2}
Every \gener\ c.e.~theory $T$ is $\cothe T$-creative. 
\end{thm}

\begin{proof}
Let $T$ be a $\tau$-\gener\ c.e.~$\mc L$-theory and let $\W{i}$ be any c.e.~set such that $T_{\mathfrak{p}} \cap \W{i} = \emptyset$. 
Let $\alpha_i$ be a well-behaved 1-$\Sigma_1$-formula that represents $\W i$ and let
${\sf Pr}_T$ be a well-behaved 1-$\Sigma_1$-formula that represents provability in $T$.
By Theorem~\ref{puntneussmurf}, we can effectively find a $\Sigma_1$-sentence $\jmath$ from $i$ satisfying the equivalence
$\RR_0 \vdash \jmath \leftrightarrow \alpha_i({\certt{\jmath}^\tau})  \leq \PR_T({\certt{\jmath}^\tau}).$\footnote{As before, we omit G\"odel numbering brackets.
Note that, as an intermediate step, we
have to find a 1-$\Sigma_1$-formula $\sigma(x)$ such that $\RR_0 \vdash \sigma(\gn{\beta}) \iff \PR_T(\gn{\certt{\beta}^\tau})$.} 

Suppose $T \vdash \certt{\jmath}^\tau$. 
Then, we have $\N \models \PR_T({\certt{\jmath}^\tau}) < \alpha_i(\certt{\jmath}^\tau)$ because $T_{\mathfrak{p}} \cap \W{i} = \emptyset$. 
By Theorem~\ref{ijverigesmurf2}, $\certt{\jmath}$ is inconsistent. 
Then, we have $T \vdash \neg\, \certt{\jmath}^\tau$. 
This is a contradiction. 

Suppose $\certt{\jmath}^\tau \in \W{i}$.  
Since $T_{\mathfrak{p}} \cap \W{i} = \emptyset$, we have $\N \models \jmath$. 
By Theorem \ref{CSS1}, we find $T \vdash \certt{\jmath}^\tau$. 
Thus, $T_{\mathfrak{p}} \cap \W{i}$ is non-empty. 
This is a contradiction. 

Therefore, we obtain $\certt{\jmath}^\tau \notin T_{\mathfrak{p}} \cup \W{i}$. 
This implies that $\jmath$ is false. 
By Theorem \ref{CSS1} again, we obtain that $\certt{\jmath}^\tau \vdash T$, i.e., $\certt{\jmath}^\tau \in \cothe T$. 
Then the partial computable function $\Psi(i) : = \certt{\jmath}^\tau$ witnesses the $\cothe T$-creativity of $T$. 
\end{proof}

\begin{cor}\label{gen_Vaught2_cor}
Every \gener\ c.e.~theory $T$ is effectively half-essentially $\cothe T$-incomplete. 
\end{cor}

Consider a \gener\ c.e.~$\mc L$-theory $T$. 
If $U$ is a c.e.~$\mc L$-theory such that $T + U$ is consistent, then by Corollary~\ref{gen_Vaught2_cor}, we can effectively find a sentence $\varphi$ such that $\varphi \vdash T$, $T \nvdash \varphi$, and $U \nvdash \neg\, \varphi$. 
Then the $\mc L$-theory $A$ axiomatized by $\varphi$ is a proper extension of $T$ such that $A+U$ is consistent.
This shows that Corollary~\ref{gen_Vaught2_cor} is in fact a generalization of Theorem~\ref{Vau2}.




\begin{cor}
    Every \gener\ c.e.~theory is deductively equivalent to the intersection of the deductive closures of its finite same-signature extensions.
\end{cor}

We note that not every c.e.~theory is deductively equivalent to the intersection of the deductive closures of its finite same-signature extensions.
For example, {\sf PA} has no consistent finite same-signature extensions. So, the relevant intersection would be the inconsistent $\LA$-theory.

We can define the notions `effective $\mc X$-inseparability' and `strong effective $\mc X$-inseparability' in the forms that witnesses can be found from the set $\mc X$. 
The following theorem is a special case of a theorem proved in our paper \cite{KV24}. 

\begin{thm}[Kurahashi and Visser {\cite[Theorem 5.5]{KV24}}]\label{Thm_KV}
If $T$ is $\cothe T$-creative and effectively inseparable, then $T$ is strongly effectively $\cothe T$-inseparable. 
\end{thm}

Since $\RR_0$ is effectively inseparable, every \gener\ theory is also effectively inseparable. 
Thus, Theorems \ref{gen_Vaught2} and \ref{Thm_KV} establish the following theorem which is a generalization of Theorem \ref{Vau1} and a strengthening of Theorem \ref{gen_Vaught2}. 

\begin{thm}\label{gen_Vaught3}
For any \gener\ c.e.~theory $T$, we have that $T$ is strongly effectively $\cothe T$-inseparable. 
\end{thm}

Theorem \ref{gen_Vaught3} is the strongest form of the first incompleteness theorem given in the present paper. 
Of course, one can prove Theorem \ref{gen_Vaught3} directly in a similar way as described in the proof of Theorem \ref{Vau1} above.

We close this section with the following application of Theorem \ref{gen_Vaught3}. 

\begin{defn}
Let $T$ be a c.e.~$\mc L$-theory and $\mc X$ be a set of $\mc L$-sentences. 
We say that $T$ is \emph{effectively uniformly essentially $\mc X$-incomplete} iff there exists a partial computable function $\Phi$ such that for every computable sequence of consistent c.e.~extensions $U_i$ of $T$ with index $j$, we have that $\Phi(j)$ converges, $\Phi(j) \in \mc X$, and for all $i$, $U_i \nvdash \Phi(j)$ and $U_i \nvdash \neg\, \Phi(j)$. 
\end{defn}

We proved in \cite{KV24} the following theorem. 

\begin{thm}[Kurahashi and Visser {\cite[Theorem 2.9]{KV24}}]
Let $T$ be any c.e.~$\mc L$-theory and $\mc X$ be any set of $\mc L$-sentences. 
The following are equivalent:
\begin{enumerate}[a.]
    \item $T$ is effectively $\mc X$-inseparable. 

    \item $T$ is effectively uniformly essentially $\mc X$-incomplete. 
\end{enumerate}
\end{thm}

By combining this theorem with Theorem \ref{gen_Vaught3}, we obtain the following strengthening of Corollary \ref{gen_Vaught2_cor}. 

\begin{cor}
Every \gener\ c.e.~theory $T$ is effectively uniformly essentially $\cothe T$-incomplete. 
\end{cor}

\section{Various Facts about Degree Structures}\label{degrees}

In this section, we provide various applications of our framework to degrees of interpretability.

\subsection{Useful Insights}
We first remind the reader of a special property of $\RR$.
\begin{thm}[Visser {\cite[Theorem 6]{Vis14}}]\label{Visser2}
For any c.e.~theory $T$, we have that $\RR \TR T$ if and only if every finite subtheory of $T$ has a finite model. 
\end{thm}
We note that this property is inherited by every c.e.~theory that is mutually interpretable with $\RR$.

We have the following definition.
\begin{itemize}
\item
The theory $T$ is a \emph{globalizer} iff, for every c.e.~theory $W$, whenever $T \rhd_{\sf loc} W$, then $T \rhd W$.
\end{itemize}

Examples of globalizers are {\sf PRA}, {\sf PA}, and {\sf ZF}. Theorem~\ref{Visser2} has the following useful consequence, which also appears 
in \cite{Vis14}.

\begin{thm}\label{musmurf}
$\RR$ is a globalizer.
\end{thm}

\begin{proof}
    Suppose $\RR \rhd_{\sf loc} U$. Then, for every finitely axiomatized subtheory $U_0$ of $U$, we have $\RR \rhd U_0$. 
    So, for every finitely axiomatized subtheory $U_0$ of $U$, there is a finitely axiomatized subtheory $A$ of $\RR$,
    such that $A \rhd U_0$. Since $A$ has a finite model, so has $U_0$. We may conclude that
    every finitely axiomatized subtheory $U_0$ of $U$ has a finite model. \emph{Ergo}, $\RR \rhd U$.
\end{proof}

Cobham has shown that $\RR_0$ is mutually interpretable with $\RR$. See \cite{JS83}. It follows that the insights contained in
Theorems~\ref{Visser2} and \ref{musmurf} are inherited by $\RR_0$. 

We give two useful results.  We assume that we have a $\Sigma_1$-representation of 
interpretability $\rhd$ for the case that the interpreted theory is finitely axiomatised and the interpreting theory is computably enumerable.
For later use we also assume that this representation is well-behaved. \emph{Par abus de langage}, we write $\rhd$ both for the meta-notion 
and for its theory-internal representation.

\begin{thm}\label{loeblike}
Let $W$ be a $\tau$-\gener\ c.e.~theory and let $A$ be finitely axiomatized.
We can effectively find a $\Sigma_1$-sentence $\lambda$ from an index of $W$, such that $\N \models \lambda$, 
$\certt{\lambda}^\tau \rhd A$, and $W \rhd A$ are equivalent. 
\end{thm}

\begin{proof}
Let $W$ be $\tau$-\gener\ and let $A$ be finitely axiomatized.
By Theorem~\ref{puntneussmurf}, we obtain a $\Sigma_1$-sentence $\lambda$ satisfying the following equivalence: 
\[
    \RR_0 \vdash \lambda \leftrightarrow \certt\lambda^\tau \TR A. 
\]

Suppose $\certt\lambda^\tau \TR A$. 
Then, we have $\N \models \lambda$ and, thus, $W \vdash \certt\lambda^\tau$ by Theorem \ref{CSS1}. Hence, $W\TR A$.

Conversely, suppose that $W \TR A$. If $\N \models \neg\, \lambda $, then we have $\certt\lambda^\tau \vdash W$ by Theorem \ref{CSS1}, and, hence, 
$\certt\lambda^\tau \TR A$. 
By the fixed point equation, we find $\N \models \lambda$, contradicting our assumption that $\N \models \neg\,\lambda$.
So, we may conclude that $\N \models \lambda$, and, thus, $\certt\lambda^\tau \TR A$.
\end{proof}

\begin{rem}
The proof of Theorem~\ref{loeblike} is strongly reminiscent of the proof of L\"ob's Theorem.
Regrettably, it does not seem that we can take the further step to obtain an analogue of L\"ob's Theorem, to wit:
\[W \rhd A\text{ iff }(W+\certt{W\rhd A}^\tau) \rhd A.\] The left-to-right direction is trivial, but we do not know about the right-to-left direction at the moment.
\end{rem}

\begin{ex}
    Juvenal Murwanashyaka asked whether there is a finitely axiomatized theory $B$ that interprets 
{\sf VS} but does not interpret {\sf AS}. Theorem~\ref{loeblike} provides an example that, additionally, is a
same-signature extension of {\sf VS}.

We can see this as follows.
Since {\sf VS} is $\RR_{0{\sf p}}$-sourced, we have, by Theorem~\ref{loeblike}, a finite same-signature-theory $B$ ($= \certt{\lambda}^\tau)$, such that $B\rhd {\sf AS}$ iff ${\sf VS} \rhd {\sf AS}$. However, ${\sf VS} \nrhd {\sf AS}$, since, otherwise,
a finite subtheory of {\sf VS} would interpret {\sf AS}. Such finite subtheories are interpretable in the theory of non-surjective pairing,
i.e., ${\sf VS}0+{\sf VS}2$, and, as is well-known this theory has a decidable extension. On the other hand, {\sf AS} is essentially
undecidable.
Since $\N \models \neg \lambda$, we have that $B$ is an extension of ${\sf VS}$ by Theorem \ref{CSS1}. 

Similarly, we can specify a finitely axiomatized same-signature extension of {\sf PRA} that does not interpret 
$\mathrm I\Sigma_1$.
\end{ex}

\begin{cor}\label{useful2}
Let $W$ be a $\tau$-\gener\ c.e.~theory, and let $T$ be any c.e.~theory. 

\begin{enumerate}[i.]
    \item 
    Suppose
     $W \NTR_{\sf loc} T$. Then,
     there is a false $\Sigma_1$-sentence $\lambda$ such that $\certt\lambda^\tau \NTR_{\sf loc} T$.
     \item
     Suppose $W$ is a globalizer and $W \NTR T$.
     Then,  there is a false $\Sigma_1$-sentence $\lambda$ such that $\certt\lambda^\tau \NTR T$.
\end{enumerate}
\end{cor}

\begin{proof}
We prove (i).    
Let $W$ be $\tau$-\gener\ and let $T$ be any c.e.~theory. 
    Suppose $W \NTR_{\sf loc} T$. 
    It follows that  $W \NTR A$, for some finitely axiomatized subtheory $A$ of $T$. 
We apply Theorem~\ref{loeblike} to find a false $\Sigma_1$-sentence $\lambda$ such that $\certt{\lambda}^\tau\NTR A$.
It follows that $\certt{\lambda}^\tau \NTR_{\sf loc}T$.

(ii) is immediate from (i).
\end{proof}

Here is another result in the same spirit as Theorem~\ref{loeblike} that uses a Rosser argument.
We remind the reader of  our representation of $\rhd$ is well-behaved (under the assumption that the interpreted theory is finitely axiomatised and
the interpreting theory is c.e.).

\begin{thm}\label{rossersmurf}
Let $W$ be a $\tau$-\gener\ theory, let  $T$ be a c.e.~theory such that $T \rhd_{\sf loc} W$, and let $A$ be finitely axiomatized.
Then, there is a $\Sigma_1$-sentence $\theta$, which is $\RR_0$-provably equivalent to
a witness comparison sentence, such that the following are equivalent:

\begin{enumerate}[a.]
\item
$((T \owedge \certt{\theta}^\tau)\TR A)$ or $(T \TR (A \ovee \certt{\theta}^\tau))$,
\item
$ T \TR A$.
\end{enumerate}
\end{thm}

\begin{proof}
Suppose  $T \TR_{\mathsf{loc}} W$.  
    By Theorem~\ref{puntneussmurf}, we obtain a $\Sigma_1$-sentence $\theta$  satisfying the following equivalence:
   \[\RR_0 \vdash \theta \leftrightarrow ((T \owedge \certt{\theta}^\tau) \TR A) \bleq (T \TR (A \ovee \certt{\theta}^\tau)).\]

Clearly (b) implies (a). We show that (a) implies (b).
Suppose we have (a). Let 
$\eta :=((T \owedge \certt{\theta}^\tau) \TR A) \leq (T \TR (A \ovee \certt{\theta}^\tau)) $.
Clearly, we have $\N \models \eta$ or $N \models \eta^\bot$.

    Suppose $\N \models \eta$. It follows that $\N \models \theta$ and $(T \owedge \certt{\theta}^\tau)\rhd A$. 
    It follows that 
    $W \vdash \certt{\theta}^\tau$, by Theorem~\ref{CSS1}. Since $T \rhd_{\sf loc} W$, we find $T \rhd \certt{\theta}^\tau$.
    Hence, $T \rhd (T\owedge \certt{\theta}^\tau) \rhd A$.

    Suppose $\N \models \eta^\bot$. It follows that $\mathbb N \models (T \TR (A \ovee \certt{\theta}^\tau)) \bles 
    ((T \owedge \certt{\theta}^\tau) \TR A)$.
    By Theorem~\ref{ijverigesmurf2}, $\certt{\theta}$ is inconsistent, so
     $\certt{\theta}^\tau$ is inconsistent.
It also follows from $\N \models\eta^\bot$ that $T \TR (A \ovee \certt{\theta}^\tau)$.
Hence, $T\rhd A$.
\end{proof}

The following theorem, which is in particular the case where $T$ and $A$ are same-signature-theories of $W$, can be proved in the same way.

\begin{thm}\label{rossersmurf2}
Let $W$ be a $\tau$-\gener\ theory, let $T$ be a c.e.~theory in the signature of $W$ such that $T \supseteq W$, and let $A$ be finitely axiomatized theory in the signature of $W$. 
Then, there is a $\Sigma_1$-sentence $\theta$, which is $\RR_0$-provably equivalent to a
witness comparison sentence, such that the following are equivalent:
\begin{enumerate}[a.]
\item
$((T + \certt{\theta}^\tau)\TR A)$ or $T \TR B$,
where $B = \{\phi \lor \certt{\theta}^\tau \mid \phi \in A\}$.
\item
$ T \TR A$.
\end{enumerate}
\end{thm}

\subsection{Applications of Certified Extension}

We turn to the consideration of various density results.

\begin{thm}\label{App4}
Consider c.e.~theories $T$ and $S$ such that $S \NTR_{\mathsf{loc}} T$.
Then, there exists a c.e.~theory $U$ such that $T \TR U$ and $S \NTR_{\mathsf{loc}} U$ and $U \NTR_{\mathsf{loc}} T$. 
Moreover, if $T \rhd S$, then $U \rhd S$ and, if $T \rhd_{\sf loc} S$, then $U \rhd_{\sf loc} S$.
\end{thm}
\begin{proof}
    Suppose  $S \NTR_{\mathsf{loc}} T$. 
    Then, we can find a finite subtheory $A$ of $T$, such that $S \NTR A$. 
    Since $S \rhd_{\sf loc} \RR_0$ and $\RR_0$ is a $\tau$-\gener\ theory, we can apply Theorem~\ref{rossersmurf} to $S$ and $A$. Let
    $\theta$ be the promised $\Sigma_1$-sentence.
    Since $S \NTR A$, we find $(S\owedge \certt{\theta}^\tau) \NTR A$ and $S \NTR (A \ovee \certt{\theta}^\tau)$.
    
Let $U : = (S \owedge \certt{\theta}^\tau) \ovee T$. It is immediate that $T \rhd U$. Moreover,
it is also immediate that, if $T\rhd S$, then $U\rhd S$ and, if $T\rhd_{\sf loc} S$, then $U \TR_{\sf loc} S$. 

Suppose $S \TR_{\sf loc} U$. Then, $S \TR (A \ovee \certt{\theta}^\tau)$. \emph{Quod non}.
Suppose $U \TR_{\sf loc} T$. Then, $ (S\owedge \certt{\theta}) \TR A$. \emph{Quod non}.
    \end{proof}

    \begin{prob}
        The proof of Theorem~\ref{App4}, seems to use specific properties of $\RR_0$. Is there a good way to generalize it?
    \end{prob}

We have immediately the following corollaries.

\begin{cor}\label{locdens}
Consider c.e.~theories $S$ and $T$. Suppose $T \lhdnneq_{\sf loc} S$. Then, there exists a c.e.~theory $U$ such that $T \lhdnneq_{\sf loc} U \lhdnneq_{\sf loc} S$. 
\end{cor}

The density of the degrees of local interpretability of c.e.~theories was first proved by Jan Mycielski, Pavel Pudl\'ak and Alan Stern in their
classical paper~\cite[Corollary 6.17]{myci:latt90}.

\begin{cor}\label{density1}
Consider c.e.~theories $S$ and $T$, where either $S$ is a globalizer or
$T$ is finitely axiomatized. Suppose $T \lhdnneq S$. Then, there exists a c.e.~theory $U$ such that $T \lhdnneq U \lhdnneq S$. 
\end{cor}

\begin{proof}
We note that if either $S$ is a globalizer or $T$ is finitely axiomatized, then $S\rhd_{\sf loc} T$ iff $S \rhd T$.
\end{proof}

\begin{ex}\label{lazysmurf}
    Consider {\sf INF} the theory in the signature of identity with axioms saying, for each $n$, that there are at least $n$ elements  and
    {\sf TWO} the theory in the signature of identity with an axiom saying that there are precisely two elements. 
    Then, we have ${\sf TWO} \lhdnneq {\sf INF}$. Every theory that has a finite model is interpretable in {\sf TWO} and every theory that
    has only infinite models proves {\sf INF}. So, there can be no theory strictly $\lhd$-between {\sf TWO} and {\sf INF}. \emph{Ergo}, 
    density fails in general in the degrees of interpretability of c.e.~theories. 
\end{ex}

\begin{prob}
Example~\ref{lazysmurf} seems too easy. What if we do have density for all theories with no finite models?
So, it would be good to have some further classes of examples.
\end{prob}

In one special case, we can constrain the in-between theories a bit more. The following theorem is
a generalization of \cite[Theorem 2]{Vis17}.

\begin{thm}\label{App3}
Let $W$ be a $\tau$-\gener\ c.e.~theory. 
Consider c.e.~theories $S$ and $T$ such that $W \subseteq S \subseteq T$ and $S \NTR_{\sf loc} T$. 
Then, there exists a c.e.~theory $U$ such that $S \subseteq U \subseteq T$ and $S \NTR_{\sf loc} U$ and $U \NTR_{\sf loc} T$.
\end{thm}
\begin{proof}
Suppose $W$ is $\tau$-\gener, $W \subseteq S \subseteq T$ and $S \NTR_{\sf loc} T$.
Then, we can find a finite subtheory $A$ of $T$, such that $S \NTR A$. 
We apply Theorem~\ref{rossersmurf2} to $W$, $S \supseteq W$, and $A$. 
Let $\theta$ be the promised $\Sigma_1$-sentence.
Since $S \NTR A$, we find $(S + \certt{\theta}^\tau) \NTR A$ and $S \NTR B$, where $B = \{\phi \lor \certt{\theta}^\tau \mid \phi \in A\}$.
    
Let $U : = \{\phi \lor \psi \mid \phi \in  T$ and $\psi \in (S + \certt{\theta}^\tau)\}$. It is immediate that $S \subseteq U \subseteq T$. 

Since $B$ is a finite subtheory of $U$ and $S \NTR B$, we have that $S \NTR_{\sf loc} U$. 
Also since $U \subseteq (S + \certt{\theta}^\tau)$ and $(S + \certt{\theta}^\tau) \NTR A$, we have $U \NTR_{\sf loc} T$. 
\end{proof}

\begin{cor}\label{density2}
Suppose that $W$ is a $\tau$-\gener\ c.e.~theory. 
Consider c.e.~theories $S$ and $T$ such that $W \subseteq S \subseteq T$, where either $S$ is a globalizer or
$T$ is finitely axiomatized. Suppose $S \NTR T$. Then, there exists a c.e.~theory $U$ such that $S \subseteq U \subseteq T$ and $S \NTR U$ and $U \NTR T$.
\end{cor}

\begin{thm}\label{SFF}
Any finite theory is the supremum of the finite theories strictly below it in the lattice of the interpretability degrees of c.e.~theories. 
\end{thm}

We will give two proofs. The first uses an ordinary G\"odel numbering and the second a self-referential one.
We remind the reader that we chose the representation of $\rhd$ in such a way that, as long as the interpreted theory is finite
and the interpreting one c.e., it is well-behaved and 1-$\Sigma_1$.

\begin{proof}[Proof with ordinary G\"odel numbering.]
Suppose $A$ is a finitely axiomatized theory. If $A$ is in the minimal
degree, we are immediately done.  Suppose $A$ is non-minimal. 
Suppose $U$ interprets all finitely axiomatized theories strictly below $A$.
We have to show that $U \rhd A$.

By  Theorem~\ref{puntneussmurf}(ii), we can find $\Sigma_1$-sentences $\rho$ and $\theta$ such that:
\begin{itemize}
\item 
$\RR_0\vdash \rho \iff ((U\rhd (A\ovee [\theta])) \vee (([\rho]\ovee [\theta])\rhd A)) \bleq 
(U\rhd (A\ovee [\rho]))$.
\item 
$\RR_0 \vdash \theta \iff (U\rhd (A\ovee [\rho])) <((U\rhd (A\ovee [\theta])) \vee (([\rho]\ovee [\theta])\rhd A))$. 
\end{itemize}
Let
$\eta := (U\rhd (A\ovee [\theta])) \vee (([\rho]\ovee [\theta])\rhd A)) \bleq 
(U\rhd (A\ovee [\rho]))$.

\medskip\noindent
\emph{Claim 1:}
Suppose $U\rhd (A \ovee [\theta])$. Then, $U \rhd A$.

\medskip\noindent 
\emph{Proof of Claim 1}
Suppose $U\rhd (A \ovee [\theta])$. Then, $\N \models \eta$ or $\N \models\eta^\bot$. 
In the first case, we have, by Theorem~\ref{ijverigesmurf}, that $[\theta]\vdash \bot$. Hence, 
$U\rhd A$.

In the second case, we have $U\rhd (A\ovee [\rho])$. By Theorem~\ref{ijverigesmurf}, we find $[\rho]\vdash \bot$. 
Hence, again, $U\rhd A$. 

So, in both cases, we may conclude $U\rhd A$.

\medskip\noindent \emph{Claim 2:} 
Suppose $U\rhd (A \ovee[\rho])$. Then, $U \rhd A$.

\medskip\noindent 
\emph{Proof of Claim 2.}
Suppose $U\rhd (A \ovee[\rho])$.
 It follows that
$\N \models \eta$ or $\N \models \eta^\bot$. In the first case it follows that 
(a) $U\rhd (A\ovee [\theta])$ or  (b) $([\rho]\ovee [\theta])\rhd A$.
In subcase (a), we find, by Claim 1, that $U \rhd A$.
In subcase (b), we have $\N \models \rho$, and, hence, by Theorem~\ref{CSS}, that  $\RR_0 \vdash [\rho]$.
So, $[\rho]$ has a finite model, and, thus,
$A$ is in the minimal degree, contradicting our assumption on $A$. 

Suppose $\N \models \eta^\bot$. In that case, we have $U \rhd (A\ovee [\rho])$. So, by Theorem~\ref{ijverigesmurf}, we find $U \rhd A$. 

\medskip\noindent
\emph{Claim 3:}
We have either $(A \ovee [\rho]) \nrhd A$ or $(A \ovee [\theta]) \nrhd A$.

\bigskip\noindent
\emph{Proof of Claim 3.}
Suppose we have both $(A \ovee[\rho]) \rhd A$ and $(A\ovee [\theta]) \rhd A$. 
Then,
$([\rho]\ovee [\theta]) \rhd A$. It follows that $\N \models \eta$ or $\N \models\eta^\bot$, and, hence, that
$\N \models \rho$ or $\N \models\theta$. 
In both cases we may conclude that $([\rho]\ovee [\theta])$ has a finite model, so
$A$ is in the minimal degree, contradicting our assumption on $A$. 

\medskip We are now ready to prove the theorem. By Claim 3,
 one of $A \ovee[\rho]$ and $A\ovee [\theta]$ is strictly below $A$, and,
hence, below $U$. If $U \rhd (A\ovee [\theta])$, it  follows by Claim 1 that $U\rhd A$.
If $U \rhd (A\ovee [\rho]) $, it  follows by Claim 2 that $U\rhd A$.
So, $A$ is indeed the supremum of the finite elements
strictly below it. 
\end{proof}

We now give our proof using a self-referential G\"odel numbering. The
proof will be largely the same, only we need just one fixed point.
We note that e.g. the arithmetized form of $U \rhd A$ in this proof is the form appropriate for the  
self-referential G\"odel numbering and, thus, is different from the case
of the ordinary numbering. We opted to keep the same notations for readability's sake,
but the reader should keep the point in mind.

\begin{proof}[Proof with self-referential G\"odel numbering.]
Suppose $A$ is a finitely axiomatized theory. If $A$ is in the minimal
degree, we are immediately done.  Suppose $A$ is non-minimal. 
Suppose $U$ interprets all finitely axiomatized theories strictly below $A$.
We have to show that $U \rhd A$.
We find $\rho$ with:
\[  \rho = ((U\rhd (A\ovee [\rho^{\bot}])) \vee (([\rho]\ovee [\rho^{\bot}])\rhd A)) \bleq 
(U\rhd (A\ovee [\rho)).\]

\medskip\noindent
\emph{Claim 1:}
Suppose $U\rhd (A \ovee [\rho^{\bot}])$. Then, $U \rhd A$.

\medskip\noindent 
\emph{Proof of Claim 1}
Suppose $U\rhd (A \ovee [\rho^{\bot}])$. Then, $\N \models \rho$ or $\N \models\rho^\bot$. 
In the first case, we have, by Theorem~\ref{ijverigesmurf}, that $[\rho^{\bot}]\vdash \bot$. Hence, 
$U\rhd A$.

In the second case, we have $U\rhd (A\ovee [\rho])$. By Theorem~\ref{ijverigesmurf}, we find $[\rho]\vdash \bot$. 
Hence, again, $U\rhd A$. 

So, in both cases, we may conclude $U\rhd A$.

\medskip\noindent \emph{Claim 2:} 
Suppose $U\rhd (A \ovee[\rho])$. Then, $U \rhd A$.

\medskip\noindent 
\emph{Proof of Claim 2.}
Suppose $U\rhd (A \ovee[\rho])$.
 It follows that
$\N \models \rho$ or $\N \models \rho^\bot$. In the first case it follows that 
(a) $U\rhd (A\ovee [\rho^{\bot}])$ or  (b) $([\rho]\ovee [\rho^{\bot}])\rhd A$.
In subcase (a), we find, by Claim 1, that $U \rhd A$.
In subcase (b), we have $\N \models \rho$, and, hence, by Theorem~\ref{CSS}, that  $\RR_0 \vdash [\rho]$.
So, $[\rho]$ has a finite model, and, thus,
$A$ is in the minimal degree, contradicting our assumption on $A$. 

Suppose $\N \models \rho^\bot$. In that case, we have $U \rhd (A\ovee [\rho])$. So, by Theorem~\ref{ijverigesmurf}, we find $U \rhd A$. 

\medskip\noindent
\emph{Claim 3:}
We have either $(A \ovee [\rho]) \nrhd A$ or $(A \ovee [\rho^{\bot}]) \nrhd A$.

\bigskip\noindent
\emph{Proof of Claim 3.}
Suppose we have both $(A \ovee[\rho]) \rhd A$ and $(A\ovee [\rho^{\bot}]) \rhd A$. 
Then,
$([\rho]\ovee [\rho^{\bot}]) \rhd A$. It follows that $\N \models \rho$ or $\N \models\rho^\bot$. 
In both cases we may conclude that $([\rho]\ovee [\rho^{\bot}])$ has a finite model, so
$A$ is in the minimal degree, contradicting our assumption on $A$. 

\medskip We are now ready to prove the theorem. By Claim 3,
 one of $A \ovee[\rho]$ and $A\ovee [\rho^{\bot}]$ is strictly below $A$, and,
hence, below $U$. If $U \rhd (A\ovee [\rho^{\bot}])$, it  follows by Claim 1 that $U\rhd A$.
If $U \rhd (A\ovee [\rho]) $, it  follows by Claim 2 that $U\rhd A$.
So, $A$ is indeed the supremum of the finite elements
strictly below it. 
\end{proof}

\begin{ex}
    We note that both in the local and in the global degrees of interpretability of c.e.~theories, the degree of the theory
    {\sf INF} is not the supremum of the degrees  of the theories strictly below them, so \emph{a fortiori}, it is not the
    supremum of the degrees of the finitely axiomatizable ones.
\end{ex}

\begin{cor}
In the lattice of c.e.~degrees of interpretability, no theory $A$ can be
finitely axiomatized, non-minimal, join-irreducible, and compact. 
\end{cor}

\begin{proof}
Suppose $A$ is finitely axiomatized, non-minimal, join-irreducible, and compact.
By Theorem~\ref{SFF}, $A$ is the supremum of the finitely axiomatized theories
strictly below it. Hence, by compactness, it is 
 mutually interpretable with the supremum of a finite number of finite theories  strictly below it.
By join-irreducibility, it follows that $A$ is mutually interpretable with
a finite theory strictly below it. A contradiction.
\end{proof}


\begin{thm}\label{smurferella}
Consider a c.e.~theory $W$.
\begin{enumerate}[i.]
\item
Suppose $W$ is mutually locally interpretable with a  \gener\ theory. Then,
in the degrees
of local interpretability of c.e.~theories,
$W$ is the infimum of the finitely axiomatized theories above it.
\item
Suppose $W$ is mutually interpretable with a  \gener\ globalizer. Then,
in the degrees
of interpretability of c.e.~theories,
$W$ is the infimum of the finitely axiomatized theories above it.
\end{enumerate}
\end{thm}

\begin{proof}
We first prove (i).
It is clearly sufficient to prove the result for the case that $W$ is a $\tau$-\gener\ theory, for some $\tau$. 
    Suppose all finitely axiomatized theories that interpret $W$ locally interpret $T$. We want to show that $W\rhd_{\sf loc} T$.
    Suppose, in order to obtain a contradiction, that $W \NTR_{\sf loc} T$. Let $\lambda$ be the sentence provided by Corollary~\ref{useful2}(i) 
    such that $\certt\lambda^\tau\NTR_{\sf loc} T$.
    Since $\lambda$ is false, we have $\certt\lambda^\tau \rhd W$ and, hence, $\certt\lambda^\tau\rhd_{\sf loc} T$. A contradiction.

    We turn to (ii). Again is sufficient to prove the result for the case that $W$ is a $\tau$-\gener\ theory, for some $\tau$. 
    Suppose all finitely axiomatized theories that interpret $W$  interpret $T$. We want to show that $W\rhd T$.
    Suppose, in order to obtain a contradiction, that $W \NTR T$. Let $\lambda$ be the sentence provided by Corollary~\ref{useful2}(ii) such that $\certt\lambda^\tau\NTR T$.
    Since $\lambda$ is false, we have $\certt\lambda^\tau \rhd W$ and, hence, $\certt\lambda^\tau\rhd T$. A contradiction.
\end{proof}

\begin{cor}
Suppose $W$ is mutually interpretable with a c.e.~sequential globalizer. Then, $W$ is the interpretability infimum of all 
finitely axiomatized theories above it \textup(w.r.t.~$\lhd$\textup).
\end{cor}

\begin{proof}
Any sequential globalizer $U$ is mutually interpretable with a restricted sequential theory, to wit $\mho_U$, which is, of course, itself a globalizer.
See \cite{viss:inte18}. By Corollary~\ref{smartsmurf}, the theory $\mho_U$ is a \gener. We now apply Theorem~\ref{smurferella}.
\end{proof}

So, e.g., {\sf PA} is the infimum in the degrees of local interpretability of c.e.~theories of the finitely axiomatized theories that locally
interpret it.

\section{Conclusions}

We  presented the following two new methods.
\begin{itemize}
    \item \emph{Certification of $\Sigma_1$-witnesses:} We introduced the notion of the certification of an element (Definition \ref{certification}), and explored some consequences of the certification. The
    Certified Extension Theorem on $\RR_0$ (Theorem \ref{CSS}) is one of the main results of this study. 

    \item \emph{\gener\ theories:} We developed a generalization of the argument concerning $\RR_0$ to \gener\ theories, 
    which allows the Generalized Certified Extension Theorem (Theorem \ref{CSS1}) to be applied, 
    for example, to Vaught's set theory $\mathsf{VS}$ (Theorem \ref{VS}).
\end{itemize}

Our two methods have the following two applications: 

\begin{itemize}
    \item Certified $\Sigma_1$-sentences can be successfully applied to provide proofs of Vaught's two theorems. 
    Furthermore, we proved the strong effective $\cothe T$-inseparability of \gener \, c.e.~theories, which yields Vaught's two theorems (Theorem \ref{gen_Vaught3}). 

    \item Certified $\Sigma_1$-sentences can also be applied to the study of the degrees of interpretability of theories. 
    We proved some density results (Corollaries \ref{density1} and \ref{density2}) and studied sufficient conditions for a theory to be the supremum of the theories below it or the infimum of the theories above it (Theorems \ref{SFF} and \ref{smurferella}). 
\end{itemize}

In our paper \cite{KV24}, we specifically discussed topics related to the first application. 
This paper may be read in connection with the present paper. 

\section*{Acknowledgments}

This work was partly supported by JSPS KAKENHI Grant Numbers JP19K14586 and JP23K03200. We thank the referee for his/her 
valuable suggestions.

\bibliographystyle{alpha}
\bibliography{reference}

\end{document}